\documentclass[11pt,reqno]{amsart}

\usepackage[T1]{fontenc}
\usepackage{amsmath}							
\usepackage{amssymb}
\usepackage{amsthm}
\usepackage{amscd}
\usepackage{amsfonts}
\usepackage{mathabx}
\usepackage{stmaryrd}
\usepackage[all]{xy}

\usepackage{caption}
\usepackage{subcaption}
\usepackage{euler}

\usepackage{extarrows}

\usepackage[colorlinks, linktocpage, citecolor = purple, linkcolor = blue]{hyperref}
\usepackage{color}

\usepackage{tikz}									
\usetikzlibrary{matrix}
\usetikzlibrary{patterns}
\usetikzlibrary{positioning}
\usetikzlibrary{decorations.pathmorphing}
\usetikzlibrary{cd}

\usepackage{ytableau}
\usepackage{fullpage}
\usepackage[shortlabels]{enumitem}

\newtheorem{theorem}{Theorem}[section]
\newtheorem{lemma}[theorem]{Lemma}
\newtheorem*{lemma*}{Lemma}

\newtheorem*{corollary*}{Corollary}

\theoremstyle{definition}
								
\newtheorem{definition}[theorem]{Definition}
\newtheorem{example}[theorem]{Example}
\newtheorem*{example*}{Example}

\newtheorem{remark}[theorem]{Remark}
\newtheorem*{remark*}{Remark}

\newcommand{\ZZ}{\mathbb{Z}}

\newcommand{\PP}{\mathbb{P}}

\newcommand{\QQ}{\mathbb{Q}}
\newcommand{\RR}{\mathbb{R}}

\DeclareMathOperator{\Ker}{Ker}

\DeclareMathOperator{\Sym}{Sym}

\newcommand{\calA}{\mathcal{A}}

\newcommand{\calD}{\mathcal{D}}

\newcommand{\wt}[1]{\widetilde{#1}}

\DeclareMathOperator{\Pic}{Pic}

\DeclareMathOperator{\Prin}{Prin}
\DeclareMathOperator{\trop}{trop}
\DeclareMathOperator{\val}{val}
\DeclareMathOperator{\Trop}{Trop}

\DeclareMathOperator{\Div}{Div}

\DeclareMathOperator{\Prym}{Prym}

\let\ddiv\relax
\DeclareMathOperator{\ddiv}{div}

\DeclareMathOperator{\Jac}{Jac}

\title{Chip-firing games, Jacobians, and Prym varieties}

  \author{Yoav Len}
  \address{Mathematical Institute, University of St Andrews, St Andrews KY16 9SS, UK}
  \email{\href{mailto:yoav.len@st-andrews.ac.uk}{yoav.len@st-andrews.ac.uk}}

\begin{document}

\begin{abstract} 
We present a self-contained introduction to the theory of chip-firing games on metric graphs,  as well as the more recent theory of tropical Prym varieties.
We briefly discuss the connection between these notions and their algebraic counterparts and suggest various avenues for future research. 



\end{abstract}

\maketitle

 \setcounter{tocdepth}{2}
 \tableofcontents

\newpage

\section{Introduction}
These notes originate from the  Cambridge summer school on combinatorial algebraic geometry in September 2022. They cover chip-firing games, tropical Jacobians, and tropical Prym varieties, with a light introduction to Brill--Noether theory. The prerequisites are basic knowledge in graph theory and familiarity with standard concepts such as groups, quotient spaces, topological spaces, and metric spaces. 
Background in algebraic geometry is not necessary, however,  familiarity with the topic would help motivate  some of the notions encountered in the course, as many of whom originate from algebraic geometry. 

The theory of tropical chip-firing games is the combinatorial version of the theory of divisors on algebraic curves. The role of the  curves and their divisors is played by graphs  and   configurations of playing-chips. 
In its original form, the game was played on discrete graphs. However, it is advantageous to pass to a continuous version of the game, played on metric graphs, as it better approximates the algebraic theory and is more suitable for  problems involving moduli spaces. Furthermore, while some  definitions may seem more complicated, the continuous version tends to produce nicer and more elegant results. That should not be too surprising when comparing with other areas of maths. 

The tropical theory is not merely an analogue of the algebraic theory. A process known as \emph{tropicalization} turns an algebraic curve into a graph whose combinatorial invariants encode various geometric properties of the curve. 
Most notably, Baker's specialization lemma states that the rank of divisors may only increase under this process \cite{Baker_specialization}, an observation at the heart of various recent developments 
in the geometry of curves
\cite{JensenRanganathan_BrillNoetherwithfixedgonality, FarkasJensenPayne_Kodaira, CLRW_PBN, AbreuPacini_ResolutionOfAJ}.

The latter half of these notes is dedicated to  Jacobians and Prym varieties, both of which are 
abelian groups that classify divisors on  curves or graphs. While the Jacobian classifies divisors on a single object, the Prym variety is  associated with a double cover and classify divisors that behave nicely with respect to the double cover.
The precise definition for graphs is given in  Section \ref{sec:Pryms} and the definition for algebraic curves is analogous. 

From the perspective of the curves themselves, the Prym variety is an invariant that provides additional data and another method for probing them. From a broader perspective, Prym varieties provide a fruitful source of abelian varieties that can be parameterized and examined by looking at curves. 
While Jacobians are very well understood, they only account for a $3g-3$ dimensional locus in the moduli space of abelian varieties of dimension $g$, which is of dimension $\binom{g+1}{2}$. Prym varieties account for a $3g$-dimensional locus, which is a vast improvement. Furthermore, the Jacobian locus is contained in the closure of the Prym locus so, in a sense, we can view Prym varieties as a generalization of Jacobians. 

The fact that "there are more Pryms than Jacobians" can be used to establish structural results on the moduli space of abelian varieties. 
For instance, in dimension $5$, the numbers $\binom{g+1}{2}$ and $3g$ coincide, so Prym varieties are full-dimensional in $\calA_5$. A close examination of the geometry of Prym varieties then leads to the conclusion that the space  $\calA_5$ is uniruled \cite[Theorem 3.3]{DonagiUnirationality}. 
On a different note,  the intermediate Jacobian of a smooth Fano 3-fold $X$ is a Prym variety \cite[Section 3.2]{FarkasPrym}, and $X$ is rational iff it is an actual Jacobian. 
Clemens and Griffiths showed that this Prym variety is never a Jacobian when $X$ is a smooth cubic threefold, thus proving that smooth cubic threefolds are non-rational \cite{ClemensGriffiths_3folds}. 
The fact that intermediate Jacobians are Prym varieties is also used by Sacca, Laza, and Voisin in  \cite{LSV_Kahler} to construct and compactify hyperk\"ahler manifolds.

These notes are organized as follows.  Section \ref{sec:metric}  introduces the theory of chip-firing games on metric graphs, as well as basic notions such as the Jacobian and reduced divisors. In Section \ref{sec:JacobianStructure} we delve deeper into the Jacobian and explore its structure. The section is recommended  even for readers familiar with metric chip-firing, as the approach taken later   when  discussing Prym varieties will mirror some of the techniques used in Section \ref{sec:JacobianStructure}. Section \ref{sec:Pryms} begins the study of Prym varieties, followed by a more meticulous study of their structure in Section \ref{sec:PrymStructure}. Some of the proofs throughout are left as guided exercises. 

Appendices \ref{sec:chipFiringExercises}, \ref{sec:rankExercises}, and \ref{sec:PrymExercises} are devoted to exercises. All but 
Exercise \ref{sec:chipFiringExercises}.\ref{exer:algebraicCurve} should be solvable using only the material encountered in these notes.  Exercise \ref{sec:chipFiringExercises}.\ref{exer:algebraicCurve} requires basic familiarity with divisors on algebraic curves.
Finally, Appendix \ref{sec:openProblems} includes a variety of open problems whose solution, I believe, could be publishable. Most of them should be  solvable using only the material of this course, although one cannot know for certain before actually trying to solve them. 

\subsection*{Acknowledgements}  I warmly thank Navid Nabijou and Luca Battistella for organizing the summer school and putting together a week full of wonderful mathematical interactions. I thank Violeta Lopez, Margarida Melo, Sam Payne, Thomas Saillez, Remy Smith, and Dmitry Zakharov for helpful comments on an older draft of these notes.
I also thank  the participants of the summer school for many engaging questions and discussions.

\section{Chip-firing on metric graphs}\label{sec:metric}

\subsection{Metric graphs}
The fundamental objects studied throughout these notes are metric graphs:   metric spaces obtained from discrete graphs by assigning a real positive length to each edge. When an edge has  length $\ell$, we can identify it with the interval of length $\ell$. 
If a metric graph $\Gamma$ was obtained from a discrete graph $G$, we say that $G$ is a \emph{model} for $\Gamma$. Note that our discrete graphs are allowed to have loops and multiple edges. Unless stated otherwise, we assume that our graphs are connected. 
The \emph{genus} of a graph is the number of independent cycles, or equivalently
\begin{equation}\label{eq:genus}
g(\Gamma) = g(G) = e-v+1,
\end{equation}
where $e$ and $v$ are the number of edges and vertices respectively, and $f$ is the number of connected components (note that the genus  formula does not require the graph to be connected). The genus does not depend on the choice of model or metric.

\begin{example}
The two metric graphs of genus $3$ shown in Figure \ref{fig:twoMetricGraphs} both have the same minimal model but different edge lengths. We can obtain additional models for them by considering arbitrary points in the interior of edges as vertices. 
\begin{figure}
\centering
\begin{subfigure}{.5\textwidth}
  \centering
  \includegraphics[width=.4\linewidth]{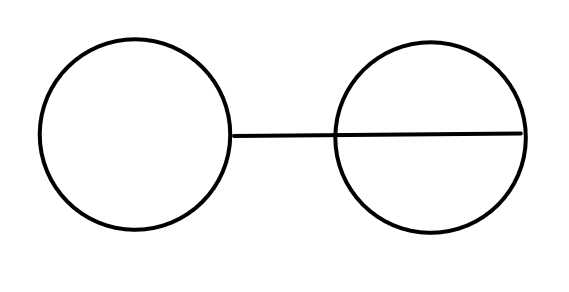}
\end{subfigure}%
\begin{subfigure}{.5\textwidth}
  \centering
  \includegraphics[width=.4\linewidth]{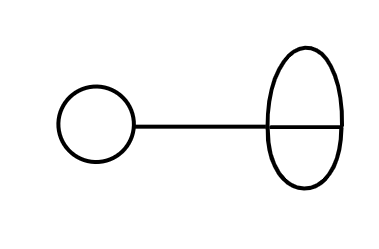}
\end{subfigure}
\caption{Two metric graphs having the same model}
\label{fig:twoMetricGraphs}
\end{figure}
\end{example}

If $g(\Gamma)\neq 1$, then $\Gamma$ has a unique minimal model $G_0$ which does not include any vertices of valency $2$. In the special case where $\Gamma$ is just a cycle, there are infinitely many minimal models consisting of a single vertex and a loop. 
We will often omit the model when clear from context. By abuse of notation,  segments of a metric graph $\Gamma$ may be referred to as edges and points as vertices. 

\subsection{Divisors and chip-firing}
Let $\Gamma$ be a metric graph. A \emph{divisor} on $\Gamma$ is a function $D:\Gamma\to\ZZ$ with finite support. We intuitively think of a divisor as assigning a finite number of playing chips at points  of the graph, where a negative number of chips is allowed. Negative chips are referred to as \emph{anti-chips}. The \emph{degree} of a divisor is the total number of chips, namely $\sum_{p\in\Gamma} D(p)$. A divisor is called \emph{effective} if it is non-negative at all points of the graph. 

\medskip

We stress that chips may be placed anywhere on the graph, not just the vertices.  

\begin{example}
Figure \ref{fig:divisor} shows a metric graph with a divisor of degree $-1$. We will often use full disks to  represent a  positive number of chips, whereas empty circles represent anti-chips. When the number of chips is not mentioned, a disk represents a single chip and an empty circle represents a single anti-chip.

\begin{figure}
    \centering
    \includegraphics[width=.2\linewidth]{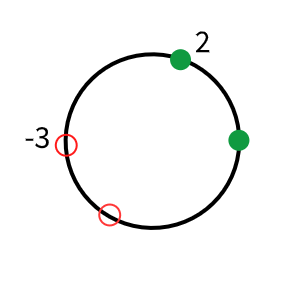}
    \caption{A divisor of degree $-1$ on a metric graph}
    \label{fig:divisor}
\end{figure}
\end{example}

The set of divisors forms an abelian group via addition, denoted $\Div(\Gamma)$.  There is a partial ordering on this set given by $D'\geq D$ whenever $D'(p)\geq D(p)$ at every point $p$. In particular, a divisor $D$ is effective if and only if $D\geq 0$. 
The subset of $\Div(\Gamma)$ consisting of divisors of degree $d$ is denoted $\Div^d(\Gamma)$. Note that $\Div^d(\Gamma)$ is a group when $d = 0$ and only a torsor of $\Div^0(\Gamma)$ otherwise. 
There is a non-canonical bijection between $\Div^d(\Gamma)$ and $\Div^0(\Gamma)$ given by $D\mapsto D - d\cdot p$, where $p$ is any point of $\Gamma$.

We will now define a certain equivalence relation on divisors that reflects the geometry of the graph (secretly, we also want it to mimic linear equivalence from algebraic geometry). Let 
\[
\varphi:\Gamma\to\RR
\]
be a continuous piecewise linear function with integer slopes. Then $\varphi$ induces a divisor $\ddiv(\varphi)$ where the value of  $\ddiv(\varphi)(p)$ at a point $p$ is the sum of incoming slopes of $\varphi$ at $p$. In other words, given a point $p$ consider all the intervals emanating from $p$. Since $\varphi$ is piecewise linear, those intervals can be chosen small enough so that $\varphi$ has a constant slope on each of them. We now take the sum of those slopes oriented towards $p$. 

\begin{example}\label{ex:divisors}
For the leftmost graph of Figure \ref{fig:piecewise_linear}, let
 $\varphi$ be the function that is constantly $0$ to the left of $u$, has slope $1$ on the segment between $u$ and $v$, and constant value $\varphi(v)$ for any $x$ to the right of $v$. Then $\ddiv(\varphi) = v - u$. 
 
 For the  graph in the middle of the figure, let $\psi$ be the function whose value is $0$ on the segment between $a$ and $b$, slope $1$ on the segment from $a$ to $d$ and on the segment from $b$ to $c$, and is constant on the segment between $c$ and $d$ (note that the function is continuous and well defined because the segments between $a$ and $d$ and between $b$ and $c$ have the same length). Then $\ddiv(\psi) = d + c - b - a$. 
 
 Finally, for the rightmost graph, let $\eta$ be the function whose  value is $0$ at $\delta$, has slope $1$ on the segments leading to the outer circle, and is constant  on the outer circle. Then $\ddiv(\eta) = \alpha + \beta + \gamma - 3\delta$.

\begin{figure}
    \centering
    \includegraphics[width=.4\linewidth]{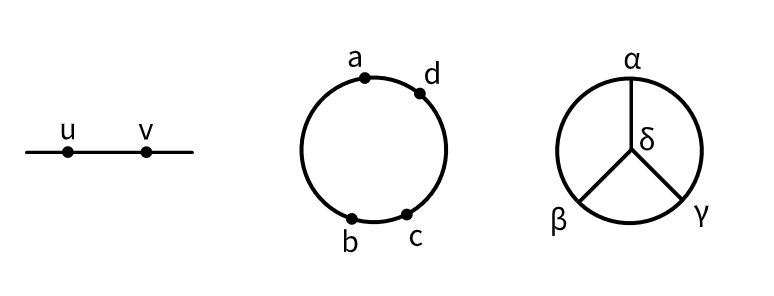}
     \caption{}
    \label{fig:piecewise_linear}
\end{figure}

\end{example}

A divisor of the form $\ddiv\varphi$ is called \emph{principal}. Divisors $D$ and $D'$ are said to be \emph{linearly equivalent}, denoted $D\simeq D'$, if $D-D' = \ddiv\varphi$ for some $\varphi$. The set of principal divisors is denoted $\Prin(\Gamma)$. Note that every principal divisor has degree $0$.  The \emph{linear system} of $D$, denoted $|D|$, is the set of effective divisor equivalent to it, namely $|D| = \{E\in\Div(\Gamma) \vert E\geq 0, E\simeq D\}$.

Modding out the group of divisors by linear equivalence gives rise to the \emph{Picard group} of the graph, namely,
\[
\Pic(\Gamma) = \Div(\Gamma)/\Prin(\Gamma).
\]
For any integer $d$, we have a group
$\Pic^d(\Gamma) = \Div^d(\Gamma)/\Prin(\Gamma)$ which classifies divisor classes of degree $d$. In the special case where $d=0$, the Picard group is known as the \emph{Jacobian} $\Jac(\Gamma)$. Note that for a fixed point $p$ and an integer $d$, we have 
$D-d\cdot p\simeq D'-d\cdot p$ whenever $D\simeq D'$. As a result, there is a bijection between the Jacobian and each Picard group $\Pic^d(\Gamma)$.

\begin{example}
Let $\Lambda$ be a line segment and let $D = a_1 v_1 + \cdots a_k v_k$ be a divisor of degree $0$, where the $v_i$ are distinct points arranged from left to right. Let $\varphi$ be the piecewise linear function whose slope is 0 to the left of $v_1$ and increases by $a_i$ at every point $v_i$. Then $\ddiv(\varphi) = D$ (the fact that $\deg(D)=0$ was used so the the slope of $\varphi$ is 0 to the right of $v_k$). 
In particular, each divisor of degree $0$ on this graph is principal and $\Jac(\Lambda)=\{0\}$.   Note that the only effective divisor equivalent to $D$ is the $0$ divisor, so the linear system of $D$ is $|D| = \{0\}$.

\end{example}

\subsection{How to actually compute linear equivalence}
As we saw in Example \ref{ex:divisors}, on a cycle graph, moving two chips in opposite directions at equal speed results in equivalent divisors. This phenomena can be generalized. Suppose that a  divisor $D'$  is obtained from a divisor $D$ by continuously moving  chips, as long as the following condition is maintained throughout the process:

\begin{equation}
  \tag{$\star$}\label{conditionA}
  \parbox{\dimexpr\linewidth-4em}{%
    \strut
    The total momentum of the chips moving  along each cycle of the graph is 0.
    \strut
  }
\end{equation}

\medskip 

\noindent Then $D$ and $D'$ are linearly equivalent. 

\begin{example}\label{ex:tree}
Suppose that $\Gamma$ has a bridge (for instance, $\Gamma$ could be the graph in Figure \ref{fig:twoMetricGraphs}) and $p$ and $p'$ are points  in its interior. Then $p\simeq p'$. Indeed, since there is no cycle containing either $p$ or $p'$,  condition \ref{conditionA} is vacuously true. 
Similarly, if $T$ is a tree, then any two divisors of the same degree are linearly equivalent. 
\end{example}

\begin{example}\label{ex:cycle}
Suppose that $\Gamma$ is the theta graph (also known as the binary graph of genus $2$) consisting of two vertices and three edges (see Figure \ref{fig:binary}), and suppose that $D$ is a divisor of degree 3 having   a single chip on each edge. Then moving each chip a small distance in the same direction results in an equivalent divisor. Indeed,  any cycle of the graph contains two of the chips, and their momentum cancels out since they are moving in opposite directions with respect to the cycle. 

On the other hand, if $D$ only consists of two chips then it doesn't move at all (that is, there is no other effective divisor equivalent to $D$). To see that, consider a cycle that contains one of the chips of $D$ but not the other. Then every time the chip is trying to move, the total momentum along this cycle is non-zero. 

\begin{figure}
    \centering
    \includegraphics[width=.3\linewidth]{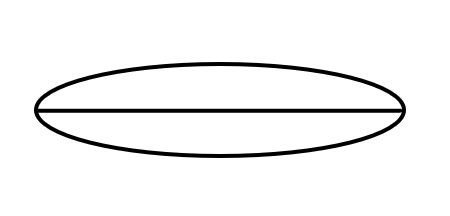}
    \caption{A binary graph of genus $2$}
    \label{fig:binary}
\end{figure}
\end{example}

\begin{example}
Let $\Gamma$ be the peace-sign graph seen on the right hand side of Figure \ref{fig:piecewise_linear}. Then $3\delta\simeq \alpha+\beta+\gamma$. Indeed moving  chips at equal speed along the edges emanating from $\delta$ satisfies Condition \ref{conditionA} and therefore results in an equivalent divisor.
\end{example}

\begin{example}\label{ex:circleGraph}
Suppose that $\Gamma$ is a cycle graph as in the middle of Figure \ref{fig:piecewise_linear}. Then, as in example \ref{ex:divisors}, given chips at $a$ and $b$, we can use condition \ref{conditionA} to move them at equal speed towards each other until they reach the points $c$ and $d$ respectively. Similarly, we can move two anti-chips towards each other so the divisor $-a-b$ is equivalent to $-c-d$. More generally, let $D = a_1 v_1 + \cdots a_k v_k$ be a divisor of degree $0$ (so some $a_i$s will inevitably be negative) and fix any point $p$ on $\Gamma$. Using condition \ref{conditionA}, we may move pairs of chips or anti-chips until one of them reaches $p$. Repeating this process, we can make sure that there is at most a single chip away from $p$. Since $\deg(D) = 0$, it follows that this new divisor (to whom $D$ is equivalent) is of the form $q-p$ where $q$ is some point of $\Gamma$. On the other hand, any divisor class of the form $q-p$ is in the Jacobian. It follows that there is a bijection between $\Jac(\Gamma)$ and $\RR/\ZZ$ obtained by mapping a divisor class $D$ of degree $0$ to the corresponding point $q$. 

\end{example}

\subsection{Reduced divisors and Dhar's burning algorithm}
The method discussed above is useful for some cases, but a more systematic approach is required in order to deal with trickier cases. Given a divisor of degree 0, we want a way to check whether it is equivalent to 0. As it turns out, this is easy to check for certain divisors known as \emph{reduced} divisors. Furthermore, every divisor is equivalent to a unique reduced divisor, which can be found using an algorithm called \emph{Dhar's burning algorithm}. 

Before defining reduced divisors, we need to discuss the notion of firing from a set. Let $A$ be a closed subset of $\Gamma$. For $\varepsilon>0$, let $A_\varepsilon$ be the set of points at distance strictly smaller than $\varepsilon$ from $A$ (in particular, $A\subseteq A_\varepsilon$). 
Choose $\varepsilon$ small enough so that $A_\varepsilon$ doesn't include any vertices not already in $A$.
Define a piecewise linear function $\theta$ that is constantly $0$ on $A$, constantly $\varepsilon$ on the complement of $A_{\varepsilon}$, and has slope 1 along $A_\varepsilon\setminus A$ (more precisely, $\theta(x) = \min(\varepsilon,\text{dist}(x,A))$). 
Now, the divisor $D_A:=\ddiv\theta$ 
is obtained by pushing a chip a small distance $\varepsilon$ away from $A$ at any edge emanating from $A$.  We refer to it as  \emph{chip-firing away from $A$}. 
Of course, $D_A$ depends on $\varepsilon$, but we usually omit it.

\begin{example}
If $A$ is the subgraph marked in blue in Figure \ref{fig:setBurning}, then $D_A = (a'-a) + (b'-b) + (c'-c) + (d'-d) + (e'-e) + (f'-f)$.

\begin{figure}
    \centering
    \includegraphics{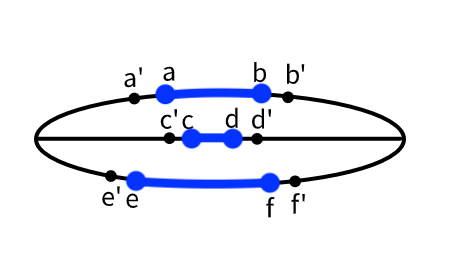}
    \caption{A metric graph with a marked subset}
    \label{fig:setBurning}
\end{figure}
\end{example}

Now fix a point $p$ of $\Gamma$. We say that a divisor $D$ is \emph{$p$-reduced} if:
\begin{enumerate}
    \item $D$ is effective away from $p$, namely $D(q)\geq 0$ for every $q\neq p$; and
    
    \item For every set $A$ such that  $p\notin A$, the divisor   $(D + D_A)(q)$ is not effective away from $p$. 
\end{enumerate}

Note that the definition does not depend at all on the value of $D$ at $p$. 
Intuitively, we think of a reduced divisor as the representative whose  chips are as close as possible to $p$. In particular, the following is true.

\begin{lemma}\label{lem:reducedInterpretation}
A divisor $D$ is equivalent to an effective divisor if and only if its $p$-reduced representative $D_p$ satisfies $D_p\geq 0$. Furthermore, if $E\simeq D$ is an effective divisor then $E(p)\leq D_p(p)$.
\end{lemma}
\begin{proof}
The proof is given as part of Exercise \ref{sec:chipFiringExercises}.\ref{exer:reduced}.
\end{proof}

So finding a  $p$-reduced representative of a divisor immediately reveals whether the divisor is equivalent to effective. But how do we do that? This is where  the \emph{Dhar's burning algorithm} comes in. 
The algorithm consists of several iterations. After every iteration we obtain an equivalent divisor that is closer to being $p$-reduced. The algorithm terminates when the divisor is $p$-reduced, which happens after a finite number of iterations. The steps are as follows.

1. Start a ``fire" from $p$ which spreads along the graph. It doesn't matter how fast the fire is spreading along each edge. The chips of $D$ act as firemen. Each chip can stop a single fire. For instance, if a vertex has 3 chips and 3 incoming fires then the fire will stop. But if another fire comes along then the vertex will succumb and the fire will spread beyond the vertex. 

2. If the fire burns through the entire graph then the divisor is $p$-reduced. Otherwise, let $A$ be the part of the graph that was not burnt. The fact that $A$ was left unburnt precisely means that the divisor $D+D_A$ is effective. In other words, $D$ is not $p$-reduced. At this point we fire from $A$. This causes chips of $D$ to move into the complement of $A$ (namely the part that was burnt). We let them move  until one of them reaches a vertex or the point $p$ (by applying $D_A$ with an appropriate $\varepsilon$).  

We now repeat the process until an iteration in which the entire graph burns. 

\begin{example}

Figure \ref{fig:Dhar} exhibits the various steps in Dhar's burning algorithm. In Figure \ref{fig:DharA}, we see a divisor of degree $5$ on the dumbbell graph with a marked point $p$. The algorithm begins with a fire spreading from $p$ in both directions until it reaches the vertex. For this stage the lengths of the edges don't play a role, and neither does the speed at which the fire is burning along each edge.  Since there are 3 `fireman' at the vertex but only 2 incoming fires, the fire is stopped. The part that was not burnt is the union of the right hand cycle and the bridge. 

We fire from that subgraph, resulting in two chips traveling from the vertex towards $p$ at equal speed along the arcs of the cycle. The lengths of the edges do play a role at this step. Since the length of the upper arc equals the length of the lower arc, the chips move until they both reach $p$, and we obtain the divisor depicted in Figure \ref{fig:DharB}. 

Performing the algorithm again, the lone fireman at the vertex cannot stop the two incoming fires. The fire therefore proceeds   until it is stopped by the heroic fireman on the bridge. Firing from the part that was not burnt, the chip on the bridge moves towards the left cycle,  resulting in the divisor shown in Figure \ref{fig:DharC}. In the next iteration, the fire is stopped by the two chips at the left vertex, and firing from the part that was not burnt we reach the divisor in Figure \ref{fig:DharD}. Now, nothing can stop the fire and it burns through the entire graph. That means that we have reached the $p$-reduced divisor.

\begin{figure}
     \centering
     \begin{subfigure}[b]{0.4\textwidth}
         \centering
         \includegraphics[width=0.6\textwidth]{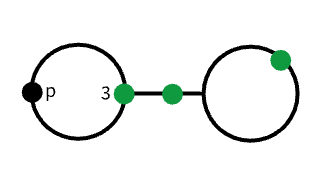}
         \caption{The initial state of the divisor}
         \label{fig:DharA}
     \end{subfigure}
     \hfill
     \begin{subfigure}[b]{0.4\textwidth}
         \centering
         \includegraphics[width=0.6\textwidth]{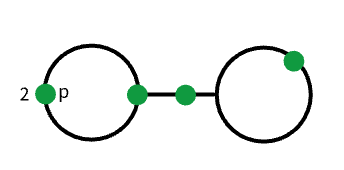}
         \caption{After the first iteration}
         \label{fig:DharB}
     \end{subfigure}
     \hfill
     \begin{subfigure}[b]{0.4\textwidth}
         \centering
         \includegraphics[width=0.6\textwidth]{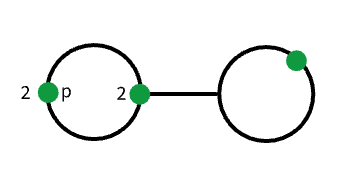}
         \caption{After the second iteration}
         \label{fig:DharC}
     \end{subfigure}
     \hfill
     \begin{subfigure}[b]{0.4\textwidth}
         \centering
         \includegraphics[width=0.6\textwidth]{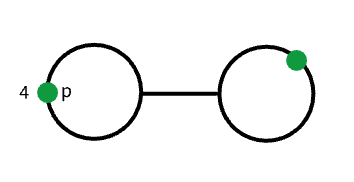}
         \caption{The $p$-reduced representative}
         \label{fig:DharD}
     \end{subfigure}
     \hfill
        \caption{An illustration of Dhar's burning algorithm}
        \label{fig:Dhar}
\end{figure}

\end{example}


\section{The rank of divisors}
As we noticed in various examples above, some divisors  tend to move more easily than others. More accurately, they have a larger linear system (Recall that the linear system of a divisor $D$, denoted $|D|$, is the collection of effective divisors linearly equivalent to $D$).  The degree of freedom of a divisor within its linear system is measured by an invariant known as the \emph{rank}. 
  It is the tropical analogue of the projective dimension of the linear system for divisors on algebraic curves.   

\begin{definition}
Let $D$ be a tropical divisor. If $|D|$ is empty then the  \emph{rank} of $D$ is $-1$. Otherwise, the rank is the largest number $r$ such that $|D-E|$ is non-empty for every effective divisor $E$ of degree $r$. 
The rank is denoted $r(D)$. 
\end{definition}

In other words, the rank is the largest number of chips that we can remove from $D$ so that it is still equivalent to an effective divisor. The quantifiers in the definition are  crucial. For instance, in order to show that $D$ has rank at least $1$, we need to show that $D-p$ is equivalent to an effective divisor for \emph{every} point $p$ of the graph. 

\begin{example}
Let $\Gamma$ be the circle graph and let $D = p_1+p_2$ for some points $p_1$ and $p_2$ of $\Gamma$. Let $q$ be any point of $\Gamma$. Using condition \ref{conditionA}, we may move the chips at $p_1$ and $p_2$ until one of them reaches $q$, so $D-q$ is equivalent to an effective divisor. Since $q$ is arbitrary, it follows that $r(D)\geq 1$.

On the other hand, let $q_1$ and $q_2$ be any points of $\Gamma$, such that the distance between $p_1$ and $q_1$ does not equal the distance between $p_2$ and $q_2$. Using Condition \ref{conditionA} again, we find that $D-q_1 - q_2$ is equivalent to $p' - q_1$ for some point $p'$ which is distinct from $q_1$. But that's as far as we can go. Using Dhar's burning algorithm (or otherwise), it is straightforward to check that $p'-q_1$ is not equivalent to any effective divisor. It follows that $D$ has rank exactly $1$.  
\end{example}

\begin{example}\label{ex:treeRank}
As we saw in Example \ref{ex:tree} any two divisors of the same degree on a tree  are equivalent to each other. So every divisor $D$ of degree $d$ is equivalent to  $p_1 + p_2 + \ldots + p_d$ for any choice  of points $p_1,\ldots,p_d$. Therefore, $D - p_1 - p_2 - \ldots - p_d$ is equivalent to an effective divisor. Since this is true for every choice of $p_1,\ldots,p_d$, it follows that $D$ has rank $d$. Note that a tree is the only graph where the rank of an effective divisor may equal its degree (see Exercise \ref{sec:rankExercises}.\ref{exer:nonTreeRank}).
\end{example}

\begin{example}
If $E$ is effective and $D$ is any divisor, then $r(D + E)\geq r(D)$. Indeed, for any choice of $p_1,\ldots, p_r$, if $D-p_1-\ldots-p_r + \ddiv(f)\geq 0$ then $D + E -p_1-\ldots-p_r + \ddiv(f)\geq 0$ as well.  Note that the rank of $D+E$ is not necessarily greater than   the rank of $D$. 
\end{example}

\begin{example}
If $\Gamma$ is any graph and $D$ has negative degree, then $r(D)=-1$. Indeed, the degree is preserved under linear equivalence, so $D$ is not equivalent to any effective divisor. 
\end{example}

One takeaway from the examples above is that computing ranks by hand can be rather time consuming. The Riemann--Roch theorem \cite{BakerNorine_RiemannRoch,MikhalkinZharkov, GathmannKerber_RiemannRoch}, which shall soon be discussed, allows for a much easier computation of the rank in some cases, or at least a good estimate in others with little effort. 
First, we need to introduce a special divisor which exists for every graph.

\begin{definition}
Let $\Gamma$ be a metric graph.  The \emph{canonical divisor} of $\Gamma$, denoted $K_\Gamma$ is defined at every point $p$ via
\[
K_\Gamma(p) = \val(p) - 2,
\]
where $\val(p)$ is the number of edges adjacent to $p$. 
\end{definition}

One can easily check that the canonical divisor of a graph of genus $g$ has degree $2g-2$ (see Exercise \ref{sec:rankExercises}.\ref{exer:canonicalDegree}).

\begin{example}
Let $\Gamma$ be the line segment on the left side of Figure \ref{fig:piecewise_linear}. Then its canonical divisor has a negative chip on the leftmost vertex and a negative chip on the rightmost vertex. 
If $\Gamma$ is the cycle graph in the centre of the figure then $K_{\Gamma}$ is simply the zero divisor. 
Finally, if $\Gamma$ is the peace sign graph on the right side of the figure then $K_{\Gamma} = \alpha + \beta +\gamma + \delta$.  
\end{example}

We  now present the main result of this section. 
\begin{theorem}[The tropical Riemann--Roch theorem]\label{thm:RiemannRoch}
Let $D$ be a divisor of degree $d$ on a graph $\Gamma$ of genus $g$. Then
\[
r(D) - r(K_\Gamma - D) = d-g+1.
\]
\end{theorem}

Here are some immediate consequences of the theorem.
\begin{example}
Let $\Gamma$ be any metric graph of genus $g$. Then $r(K_\Gamma) = g-1$. Indeed,
\[
r(K_\Gamma) - r(K_\Gamma-K_\Gamma) = (2g-2) - g +1.
\]
Since $r(0)=0$ it follows that $r(K_\Gamma) = g-1$.
\end{example}

\begin{example}
If $D$ is any divisor of degree $d$ then $r(D)\geq d-g$. Indeed, the rank of $D-K_\Gamma$ is always at least $-1$, so $r(D) \geq d-g + 1 - 1 = d-g$. 

Any divisor whose rank is greater than $d-g$ is referred to as a \emph{special} divisor.  The study of such divisors is the focus of Brill--Noether theory, which is the focus of Section \ref{sec:BrillNoether}.
\end{example}


An important consequence of  Riemann--Roch  is Clifford's theorem, which provides a general bound on the rank of divisors in terms of their degree. In the following theorem (and anywhere else), a graph that is not a tree is called \emph{hyperelliptic} if it has a divisor of degree $2$ and rank $1$. 

\begin{theorem}\label{thm:Clifford}
Let $D$ be a divisor of degree $d$ on a graph $\Gamma$ of genus $g$. If $d<2g-2$ then $r(D)\geq \frac{d}{2}$. Conversely, if $r = \frac{d}{2}$ then $\Gamma$ is hyperelliptic. 
\end{theorem}
 
\begin{proof}[Proof of the forward direction]
By the Riemann--Roch theorem, we have 
\[
r(D) - r(K_{\Gamma}-D) = d - g + 1.
\]
Since the rank of divisors is super-additive (see Exercise \ref{sec:rankExercises}.\ref{exer:superadditive}), we also have
\[
r(D) + r(K_{\Gamma}-D) \leq r(K_{\Gamma}) = g-1.
\]
Adding these two equalities we conclude that
\[
2r(D) \leq d-g+1+g-1 = d. 
\]
\end{proof}
The converse direction is referred to as the hard part of Clifford's theorem and the proof can be found in \cite{Coppens_Clifford,Len_Clifford}.  
Additional properties of  hyperelliptic graphs are dealt with in Exercises \ref{sec:rankExercises}.\ref{exer:hyperelliptic} and  \ref{sec:PrymExercises}.\ref{exer:PrymHyperelliptic}.

Arguably, the main motivation for studying the rank of tropical divisors is their connection with the rank of algebraic divisors. This is given by the following result, known as  Baker's specialization lemma  \cite[Lemma 2.8]{Baker_specialization}
\begin{lemma}\label{lem:specialization}
If $\Gamma = \trop(C)$ and $\calD = \trop(D)$, then $r(\calD)\geq r(D)$. 
\end{lemma}
In other words, the rank of divisors may only increase under tropicalization. Consequences of this lemma will be discussed in the next subsection.

\subsection{A bit of Brill--Noether theory}\label{sec:BrillNoether}

The tropical theory of divisors and their rank has been in the centre of recent developments in Brill--Noether theory,  which studies algebraic curves according to their special divisors.
In this subsection we give a short, somewhat informal, introduction to the theory.   I recommend  \cite{BakerJensen} or \cite{JensenPayne_survey} for a more detailed survey.

\begin{definition}
Let $X$ be a metric graph or an algebraic curve and fix integers $d$ and $r$. The \emph{Brill--Noether variety} of $X$, denoted $W^r_d(X)$, is the collection of divisor classes of degree $d$ and rank at least $r$ on $X$. 
\end{definition}

As the name suggests, Brill--Noether theory is concerned with the varieties $W^r_d(X)$. Here are several examples of such varieties. 

\begin{example}
As we saw in Example \ref{ex:treeRank}, if $\Gamma$ is a tree then $W^r_r(\Gamma) = \Pic^d(\Gamma)$ for any $r\geq -1$. If $\Gamma$ is not a tree then $W^r_r(\Gamma)$ is empty for any $r>0$. 
If $\Gamma$ has genus $g$ and $r < g-1$ then it follows from Clifford's theorem (Theorem \ref{thm:Clifford}) that $W^r_{2r}(\Gamma)$ is non-empty if and only if $\Gamma$ is hyperelliptic. 
\end{example}

\begin{example}\label{ex:BNg13}
Suppose that $\Gamma$ is the chain of $4$ loops whose edges have generic length. By \emph{generic} we mean that there are no obvious relation between the lengths. For instance, if they are chosen at random or if they are algebraically independent of each other. 
Then  $W^1_3$ has dimension $0$.  In other words, there are only finitely many divisors on $\Gamma$ of degree $3$ and rank $1$. 
\end{example}

As seen in Exercise \ref{sec:rankExercises}.\ref{exer:hyperelliptic}, being hyperelliptic has strong implications for the structure of the graph. 
More broadly, the moduli space  of  metric graphs is stratified according to the type of divisors that they carry.  The following result, known as the Brill--Noether theorem, characterizes the Brill--Noether varieties for large families of algebraic curves.

\begin{theorem}
Let $d\geq 1$ and $r\geq 0$. Then for any smooth curve $C$,
\[
\dim(W^r_d)(C) \geq g - (r+1)(g-d+r).
\]
Furthermore, if $C$ is general in moduli, then $\dim(W^r_d)(C) = g - (r+1)(g-d+r)$. 
\end{theorem}

By a ``general curve" we mean that, within the moduli space  $M_g$ classifying curves of genus $g$, there is an open dense set of curves where the equality holds. Put differently, if we had a way of choosing a curve at random, there would be a 100\% chance that its Brill--Noether variety had dimension $g - (r+1)(g-d+r)$.
The number 
\[
\rho(g,r,d) = g - (r+1)(g-d+r)
\]
is known as the \emph{Brill--Noether number}.  


The Brill--Noether theorem is the culmination of many different results, most notably \cite{GriffithsHarris_BrillNoether, Kempf_BrillNoether, KleimanLaskov_Special}. 
A key aspect of the proof is the observation that the moduli space of curves is irreducible and that the Brill--Noether locus is open. Therefore, in order to prove the upper bound for general curves, 
it suffices to find a single curve whose Brill--Noether variety has the correct dimension. 
It sounds like this part should be simple: we claim that a certain property should hold for almost every curve, and all we need is to find a single example, so let's just choose a curve at random. However, as it turns out, every time that we write down equations for a curve, we make some choices, so the resulting curve is not exactly random. The following quote, which I believe is due to Dave Jensen, sums the situation quite well. 
\medskip
\begin{quote}
    {\it We are searching for hay in a haystack, but all we have is a magnet. }
\end{quote}
\medskip

This is where tropical geometry can help. First, we find a graph of of genus $g$ that satisfies the statement of the Brill--Noether theorem. Namely, a graph $\Gamma$ such that $\dim(W^r_d(\Gamma))=\rho$. Next, we observe that $\Gamma$ is the tropicalization of \emph{some} algebraic curve $C$ \cite{ACP_skeleton}.
Now, the Specialization Lemma \ref{lem:specialization} implies that $C$ satisfies the statement of Brill--Noether as well.  While we don't have an explicit description of this curve $C$, we now know that it exists!

As we saw in Example \ref{ex:BNg13}, the statement of Brill--Noether holds for the chain of $4$ loops. More generally, it was shown in \cite{CoolsDraismaPayneRobeva} that, whenever $\Gamma$ is a chain of $g$ loops, where the edges have generic lengths, its Brill--Noether variety has the expected dimension $\rho$. As a consequence, we obtain a tropical proof of the celebrated Brill--Noether theorem.  Note that a tropical version the Brill--Noether theorem is only known for the chain of loops. In fact, there is an open set in the moduli of metric graphs where the statement of the Brill--Noether theorem does not hold \cite{Jensen_BNnotDense}. 

Similar techniques, along with very deep lifting techniques, were later used to find a formula for the dimension of the Brill--Noether variety for curves that are non-general in moduli \cite{Pflueger_kgonalcurves, JensenRanganathan_BrillNoetherwithfixedgonality}. Other notable results include the Maximal Rank Conjecture for quadrics \cite{JensenPayne_MRC}, that was subsequently extended to all degrees via non-tropical techniques \cite{Larson_MRC}, and a determination of the Kodaira dimension of $\mathcal{M}_{22}$ and $\mathcal{M}_{23}$ \cite{FarkasJensenPayne_Kodaira}. See also \cite{ABFS_rationalBN} for an explicit construction of general Brill--Noether curves over $\QQ$.


\section{The structure of the tropical Jacobian}\label{sec:JacobianStructure}
As previously mentioned, the tropcial Jacobian classifies divisor classes of degree 0, or  equivalently
\[
\Jac(\Gamma) = \Div^0(\Gamma)/\Prin(\Gamma).
\]

We now begin a closer investigation of its structure.  A priori, elements of the Jacobians could consist of an arbitrary number of positive and negative chips (as long as there is an equal number of each). To obtain a nice description of the Jacobian, it would be beneficial to bound that number. Furthermore, it would be great if the Jacobian had a nice geometric structure that reflected the geometry of the graph. 

\subsection{The Abel--Jacobi map}
The goals above are more easily exhibited when working in $\Pic^g(\Gamma)$ instead of $\Jac(\Gamma)$. Fix a point $p\in\Gamma$. As previously discussed, there is a bijection between  $\Pic^g(\Gamma)$ and $\Jac(\Gamma)$ by sending a divisor $D$ of degree $g$ to $D-g\cdot p$ (convince yourself that this is true!). So structural results pertaining to $\Pic^g(\Gamma)$ can be transferred to results concerning $\Jac(\Gamma)$.  The following fact, which follows from the Riemann--Roch theorem \ref{thm:RiemannRoch} will be useful. 

\begin{theorem}\label{thm:AJsurjective}
Every divisor of degree $g$ on a graph of genus $g$ is equivalent to an effective divisor. 
\end{theorem}

That's very convenient! It means that every element of $\Pic^g(\Gamma)$ can be represented as $p_1 + p_2 + \cdots + p_g$ where each $p_i\in\Gamma$. 
To make things a bit more precise, let $\Sym^d(\Gamma)$ be the symmetric product of $\Gamma$ with itself $d$ times. Points of 
$\Sym^d(\Gamma)$ are unordered $d$-tuples $\{p_1,p_2,\ldots,p_d\}$, where each $p_i\in\Gamma$. The symmetric product can naturally be identified with the set of effective divisors of degree $d$, so we often use the notation  $p_1 + \ldots + p_d$ instead. 
We are now able to define a map that is pertinent in the study of the Jacobian and its interplay with the graph.

\begin{definition}
Let $\Gamma$ be a metric graph. The \emph{Abel--Jacobi} map of degree $d$ is the map
\[
\Phi^d:\Sym^d(\Gamma)\to\Pic^d(\Gamma)
\]
sending $p_1+\cdots + p_d$ to the divisor class $[p_1+\cdots + p_d]$.
We often omit the $d$ when known from context. 
\end{definition}

\smallskip
\noindent{\bf Warning.}
The term `degree' in the definition of the Abel--Jacobi map refers to the degree of the divisors. However, the degree of the map (as a harmonic morphism of polyhedral complexes) is certainly not $d$. As we will soon see, the degree is, in fact, $1$.

\medskip
Now, Theorem \ref{thm:AJsurjective} is equivalent to  saying that the Abel Jacobi map $\Phi^g:\Sym^g(\Gamma)\to\Pic^g(\Gamma)$ is surjective. 
An analogous version of Theorem \ref{thm:AJsurjective} does not hold in lower degrees. In fact, since $\dim\Sym^d(\Gamma)=d$ and $\dim\Pic^d(\Gamma) = \dim\Jac(\Gamma)=g$, the image of the Abel--Jacobi map is not even a full dimensional subset of $\Pic^d(\Gamma)$. In other words, most divisors of degree $d$ are not equivalent to any effective divisor when $d<g$. 

\begin{example}\label{ex:dumbbell}
Consider the chain of 2 loops described in Figure \ref{fig:chainOf2Loops}, and let $D$ be any divisor of degree $2$. We claim that $D$ is equivalent to an effective divisor, thus confirming Theorem \ref{thm:AJsurjective}. First, we can assume that $D$ does not have chips in the interior of the bridge. Indeed, Condition \ref{conditionA} allows us to move chips along bridges while maintaining linear equivalence. 
Since the degree of $D$ is $2$, its restriction to one of the loops, say the left one, has degree $d$ for some $d\geq 1$. Similarly to Example \ref{ex:cycle}, using Condition \ref{conditionA}, we can move the chips of $D'$ in pairs towards the vertex, which we denote  $v$, so $D'$ is equivalent to $q + (d-1)\cdot v$. Moving the $d-1$ chips at $v$ across the bridge, the restriction to the right hand cycle has degree at least $1$. Repeating the same process on the right hand cycle, we obtain an effective divisor. 

On the other hand, consider the divisor  $q + q' - v$ of degree $1$, where $q$ and $q'$ are points in the interior of the left and right cycle respectively, and $v$ is the vertex of the left cycle. We can use Dhar's burning algorithm to see that the divisor is $v$-reduced, and therefore not equivalent to an effective divisor. Indeed, starting a fire from $v$, there are fires reaching $q$ from both directions thus burning it. There is also a fire burning through the bridge, and continuing in both directions along the right hand cycle, thus burning the chip at $q'$.

\begin{figure}
    \centering
    \includegraphics[width=.3\linewidth]{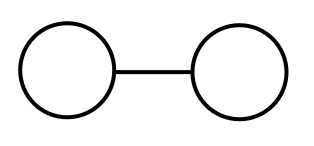}
    \caption{The chain of 2 loops}
    \label{fig:chainOf2Loops}
\end{figure}
\end{example}

The symmetric product has a natural structure of a metric space, as follows. 
For fixed edges $e_1,\ldots,e_d$ of $\Gamma$ (where repetition is allowed), the set  $I=e_1\times e_2\times\cdots\times e_d$ 
classifies ordered tuples $(p_1,p_2,\ldots,p_d)$.
 To  get a cell of the symmetric product, we need to identify points of $I$ if one is obtained from the other by permuting the coordinates. 
 We may similarly construct cells where some of the chips are restricted to vertices. 
 By performing this for each $d$-tuple of edges and vertices, and identifying  cells along their boundary where appropriate, the symmetric product naturally becomes a polyhedral complex.

Now choose a cell of $\Sym^g(\Gamma)$ induced by distinct edges  $e_1,\ldots,e_g$ of $\Gamma$, and suppose that the complement of those edges is connected. In other words, the complement of $e_1,\ldots,e_g$ is a spanning tree. 
Then one can check (for instance, using Dhar's burning algorithm, see Exercise \ref{sec:chipFiringExercises}.\ref{exer:rigidDivisors}) that, if each $p_i$ is in the interior of $e_i$, the divisor $p_1+\ldots+p_g$ is not equivalent to any effective divisor other than itself. In other words, these $g$-dimensional cells of $\Sym^g(\Gamma)$ map injectively into $\Pic^g(\Gamma)$. Cells of   $\Sym^g(\Gamma)$ that are not of this form map to lower dimensional cells of $\Pic^g(\Gamma)$. 

\begin{example}\label{ex:twoLoopsAndTheta}
Let $\Gamma$ be the chain of two loops from Figure \ref{fig:chainOf2Loops} and denote the loop edges by $e$ and $e'$ respectively. 
As we saw in Example \ref{ex:dumbbell}, every divisor of degree $2$ is equivalent to $q+q'$, where $q\in e$ and $q'\in e'$. So $\Pic^2(\Gamma)$ has a unique two-dimensional cell, namely $\Phi^2(e\times e')$. Other cells of $\Sym^2(\Gamma)$ map to lower dimensional cells of $\Pic^2(\Gamma)$. For instance, if $v$ is a vertex then $e\times v$ is $1$-dimensional, so its image is at most $1$-dimensional. But there are also $2$-dimensional cells of the symmetric product that map to $1$-dimensional cells via the Abel--Jacobi map. For instance, if $q$ and $q'$ are both in $e$ then $q+q'$ is equivalent to $v+q''$ for some $q''$ on $e$ and the vertex $v$ of $e$. Therefore, the classes $[q+q']$ are parameterized by the choice of a point $q''$. In particular, $\Phi^2(e\times e)$ is 1-dimensional.


Now let $\Gamma$ be the theta graph depicted in Figure \ref{fig:binary} with edges $e_1,e_2,e_3$. Then divisors of the form $q'+q''$  where $q'$ and $q''$ belong to distinct edges have a unique effective representative. So  the Jacobian has three maximal cells  $e_1\times e_2, e_1\times e_3$, and $e_2\times e_3$. 
\end{example}

\noindent {\bf Conclusion}
The Jacobian of a graph is tiled by cells indexed by the different spanning trees of the graph. The interior of each cell parameterizes divisors supported on the complement of the spanning tree. 

\begin{remark}
We can naturally assign a volume to each maximal cell by taking the product of the edge lengths of the corresponding edges. We conclude that the volume of the Jacobian coincides with a weighted sum over all the spanning trees. This result is known as the \emph{tropical Kirchhoff matrix-tree theorem}.  In the special case where all edge lengths are $1$, the volume of the Jacobian equals exactly the number of spanning trees. This provides a new interpretation for the classical Kirchhoff matrix-tree theorem. See \cite[Theorem 1.5]{ABKS_Canonical} for a treatment of the volume of the tropical Jacobian and \cite{GhoshZakharov_volume} for the volume of the tropical Prym. 
\end{remark}

\subsection{Jacobians and real tori}
Now that we know that the maximal dimensional cells of the Jacobian can be parameterized using tuples of edges, the question remains how they fit together. 
As it turns out, the Jacobian has a remarkably nice geometric structure. 

\begin{theorem}\cite[Theorem 3.4]{BakerFaber_tropicalAbelJacobi}\label{thm:JacobianStructure}
Suppose that the genus of $\Gamma$ is $g$. Then the Jacobian of $\Gamma$ is isomorphic, as a group, to the real torus $\RR^g/\ZZ^g$. 
\end{theorem}

This isomorphism is, in fact, a homeomorphism. Here, the topology of the Jacobian is the one induced from $\Sym^g(\Gamma)$ via the Abel--Jacobi map, so two effective divisors $E$ and $E'$ are `close' to each other when they have representatives whose chips are all close to each other.

\begin{example}
We have already seen in Example \ref{ex:cycle} that divisor classes of degree 1 on the cycle graph are given by choosing a single point on the cycle. This gives a bijection between the Jacobian and the graph itself, which is homeomorphic to  $\RR/\ZZ$. 

As we saw in Example \ref{ex:twoLoopsAndTheta}, the Jacobian of the chain of 2 loops has a unique maximal cell given by $e_1\times e_2$. Since going full circle along either of the loops brings us back to the same point, the Jacobian is, in fact, isomorphic to $\RR^2/\ZZ^2$. 
One can similarly exhibit that the Jacobian of the chain of $g$ loops or the bouquet of $g$ flowers ($g$ loops attached to a single vertex) is $\RR^g/\ZZ^g$. 

For the theta graph, we have three maximal cells $I_1\times I_2,\, I_2\times I_3$, and $I_1\times I_3$. One can check that, by gluing these cells along the boundary, we obtain a 2-torus as predicted by Theorem \ref{thm:JacobianStructure}
\end{example}

Let us see how to get a map from the Jacobian to the real torus of dimension $g$. Fix a point $p$ of $\Gamma$ and choose an oriented basis $\sigma_1,\sigma_2,\ldots, \sigma_g$ for the homology of $\Gamma$ (that is, each $\sigma_i$ is a cycle with a chosen direction). For any  point $q$  of $\Gamma$, choose a path $\gamma$ from $p$ to $q$. For each $i$ let $x_i$ be the length of the oriented intersection between $\gamma$ and $\sigma_i$. We measure this length with respect to the orientation, namely, a segment where  $\gamma$ and $\sigma_i$ are oriented in opposite directions contributes a negative length. Now, map the divisor $p-q$ to the point $(x_1,x_2,\ldots,x_g)\in\RR^g$. 

We need to be careful, because the construction as stated depends on the choice of $\gamma$, so doesn't yet give rise to a well defined map. 
However, the difference between any two such paths is given by entire copies of each basis element $\sigma_i$. Therefore, the image of $p$ in $\RR^g$ will vary by a combination of vectors of the form $u_i$ where the $j$-th coordinate of $u_i$ is the oriented intersection between $\sigma_i$ and $\sigma_j$. 

Therefore, we do get a well-defined map if we mod-out the target $\RR^g$ by the lattice generated by $u_1,u_2,\ldots, u_g$, that is, we get a map 
\[
f:\Gamma\to\RR^g/(\ZZ u_1+\ZZ u_2+\cdots +\ZZ u_g).
\]
Note that the  $u_1,u_2,\ldots, u_g$ actually form a full-rank lattice, so the target of $f$ can be identified with $\RR^g/\ZZ^g$, the real torus of dimension $g$.

We now extend $f$ by linearity to any divisor, and denote $f_d$  its restriction to divisors of degree $d$. The proof of Theorem \ref{thm:JacobianStructure}, therefore, consists of the following three ingredients. 
\begin{enumerate}
\item The map $f_0$ is well defined on divisor classes. 
\item The map $f_0$ is injective.
\item The map $f_0$ is surjective. 
\end{enumerate}

\noindent Note that $f_0$ maps  principal divisors to $0\in\RR^g/\ZZ^g$. Similarly, the divisors sent to $0$ by   $f_d$ are precisely those that are equivalent to $d\cdot p$.

\begin{example}
Let $\Gamma$ be the theta graph as in Figure \ref{fig:binary}. Denote the left vertex by $p$,  label the edges $e_1,e_2,e_3$ oriented away from $p$, and let $\ell_1,\ell_2,\ell_3$ be their lengths. Now choose a basis $\sigma_1,\sigma_2$ where $\sigma_1 = e_1 - e_2$ and $\sigma_2 = e_2-e_3$. The oriented intersection between the basis elements is 
\[
\sigma_1\cdot\sigma_1 = \ell_1 + \ell_2,
\]
\[
\sigma_1\cdot\sigma_2 = -\ell_2,
\]
\[
\sigma_2\cdot\sigma_1 = -\ell_2,
\]
\[
\sigma_2\cdot\sigma_2 = \ell_2 + \ell_3,
\]
so 
\[
\Jac(\Gamma) \simeq \RR^2/\langle(\ell_1+\ell_2,-\ell_2),(-\ell_2,\ell_2+\ell_3)\rangle.
\]
For instance, suppose that $x_1$ and $x_2$ are points on $e_1$ and $e_2$ at distance $t_1$ and $t_2$ from $p$ respectively, and let  $D=x_1+x_2$. Choosing a path from $p$ to $x_1$ along $e_1$ and from $p$ to $x_2$ along $e_2$, the divisor $D$ corresponds to the point $(t_1-t_2,t_2)$. If, on the other hand, we choose the path to $x_2$ along the edge $e_1$ then the divisor corresponds to $(t_1+\ell_1+(t_2-\ell_2), \ell_2-t_2)$. There's no need to worry though:   those two points coincide in the quotient since their difference is $(\ell_1 + \ell_2,-\ell_2)$, so the image in $\RR^2/\ZZ^2$ did not depend on the chosen paths. 
\end{example}

\begin{example}
Let $\Gamma$ be the peace-sign graph as in Figure \ref{fig:piecewise_linear} with all edge lengths equal to $1$. Choose a basis $\sigma_1,\sigma_2,\sigma_3$ where $\sigma_1$ is the path going from $\delta$ to $\alpha$, then to $\beta$ and back to $\delta$, $\sigma_2$ is going from $\delta$ to $\beta$ to $\gamma$ and back to $\delta$, and $\sigma_3$ is going from $\delta$ to $\gamma$ to $\alpha$ and back to $\delta$. 

Then the Jacobian is identified with $\RR^3/\langle u_1,u_2,u_3\rangle$, where
\[
u_1 = (3,-1,-1),
\]
\[
u_2 = (-1,3,-1),
\]
and
\[
u_3 = (-1,-,1,3). 
\]
If $x_1,x_2,x_3$ are points at distances $t_1,t_2,t_3$ respectively on the edges between $\delta$ and $\alpha,\beta,\gamma$, then the divisor $D=x_1+x_2+x_3$ corresponds to the point $(x_1-x_2,x_2-x_3,x_3-x_1)$. In the special case where $x_1=x_2=x_3$ we get $0\in\RR^3/\ZZ^3$. That is to be expected since because this divisor is equivalent to $3\delta$.
\end{example}

\begin{remark}\label{rem:Torelli}
From a topological perspective, the Jacobians of different graphs of genus $g$ are indistinguishable. To that end, we decorate the Jacobian with additional data known as a \emph{polarization}. There are several equivalent ways to define the polarization, one of which is to keep track of the \emph{theta divisor} which consists of the classes of effective divisors of degree $g-1$ (or equivalently, the image of $\Phi^{g-1}$). With this data, the Jacobian becomes a principally polarized tropical abelian variety, or PPTAV for a simple acronym that rolls off the tongue. 
We therefore have a map, known as the tropical \emph{Torelli} map,
\[
\mathcal{M}_g^{\trop} \to \mathcal{A}^{\trop}_{g}
\]
from the moduli space of tropical curves of genus $g$ to the moduli space of PPTAVs \cite{BrannettiMeloViviani_TropicalPrym}. Unlike its algebraic analogue which is injective everywhere, the  tropical Torelli map is only injective on the locus of $3$-connected graphs (see \cite{CaporasoViviani_Torelli} for additional information). 
\end{remark}

\section{Double covers and Prym varieties}\label{sec:PrymVarieties}
Prym varieties  are abelian varieties that, similarly to Jacobians, are associated with algebraic curves or metric graphs. The key difference is that they are associated with a double cover of graphs/curves rather than  a single one. 

\subsection{Harmonic covers}
We first need to discuss the tropical analogue of \'etale double covers. 

\begin{definition}
A local isometry  $\pi:\widetilde\Gamma\to\Gamma$ is called a \emph{free cover} of degree $d$ if the preimage  of every $x\in\Gamma$ consists of $d$ distinct points of $\widetilde\Gamma$. A free  cover of degree 2 is called a \emph{free double cover}. Every double cover is equipped with a non-trivial involution $\iota$ that swaps the fibres, namely  the unique map such that $\pi(x) = \iota\pi(x)$.
\end{definition}

Every free double cover of graphs is the tropicalization of an étale double cover of curves, but the converse is not true:  étale double covers can tropicalize to maps  known as \emph{harmonic double cover}.  For the sake of brevity we will restrict ourselves to free double covers which already capture much of the beautiful aspects of the theory and are, handily, already useful in applications.

\begin{example}
Figure \ref{fig:covers} shows two free double covers of the dumbbell graphs and a free double cover of the theta graph. As we shall  soon see, these graphs admit additional  free double covers.

\begin{figure}
     \centering
     \begin{subfigure}[b]{0.4\textwidth}
         \centering
         \includegraphics[width=\textwidth]{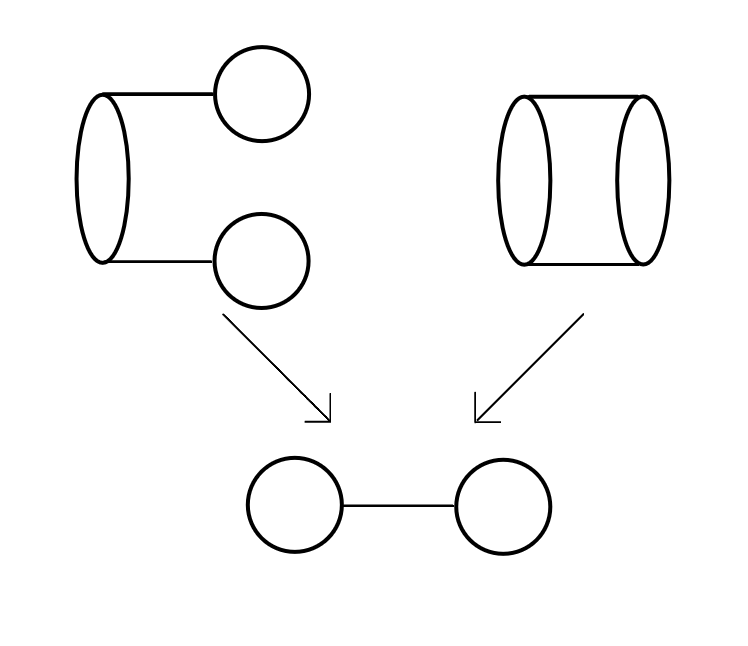}
         \caption{Two double covers of the dumbbell graph.}
         \label{fig:dumbbellCovers}
     \end{subfigure}
     \hfill
     \begin{subfigure}[b]{0.4\textwidth}
         \centering
         \includegraphics[width=0.4\textwidth]{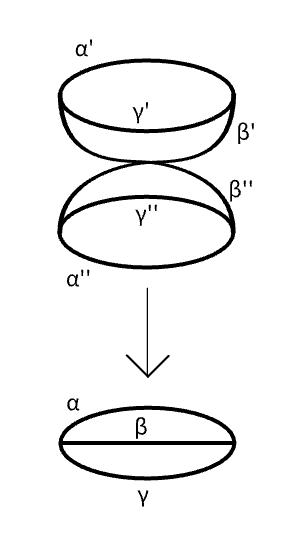}
         \caption{A double cover of the theta graph. Note that the edges $\beta'$ and $\beta''$ don't actually meet.}
         \label{fig:thetaCovers}
     \end{subfigure}
     \hfill
        \caption{Examples of free double covers}
        \label{fig:covers}
\end{figure}

\end{example}

\begin{example}\label{ex:doubleCoverCycle}
If $\Gamma$ is a cycle of length $1$ then its only connected free double cover is given by $\Gamma$ itself. If we  identify $\Gamma$ with $\RR/\simeq$, where $\theta_1\simeq\theta_2$ whenever $\theta_1-\theta_2$ is an integer, then the double cover is given by $\pi(\theta) = 2\theta$. 
\end{example}

\subsubsection{How to construct all free double covers?}

Given a metric graph $\Gamma$, begin by fixing a spanning tree $T$. Any free double cover must contain two disjoint copies of $T$, call them $T^+$ and $T^-$. Now, let $e$ be an edge of $\Gamma$ in the complement of $T$, and denote its endpoints  $v_1$ and $v_2$. Denote $v_i^{+}$ and $v_i^{-}$ the corresponding vertices in $T^+$ and $T^-$ respectively. Any double cover of $\Gamma$ must include two copies of $e$, and each copy must connect  a preimage of $v_1$ with a preimage of $v_2$. Therefore, $\widetilde\Gamma$ will either include  edges between $v_1^{+}$ and $v_2^{+}$ and between $v_1^{-}$ and $v_2^{-}$ or edges between $v_1^{+}$ and $v_2^{-}$ and between $v_1^{-}$ and $v_2^{+}$. 
 Now choose one of those two possible lifts
for each edge $e$ in $\Gamma\setminus T$ to obtain a double cover.

If the genus of $\Gamma$ is $g$ then the complement of every spanning tree consists of $g$ edges, so there are $2^g$ such double covers. Note that every double cover is constructed this way, and starting from a single spanning tree is sufficient to describe all the free double covers of a given graph.

\begin{example}
The bottom of Figure \ref{fig:spanningTrees} shows a spanning tree for the dumbbell graph (with the edges in the complement grayed out) and the top shows the two copies $T^+$ and $T^-$. In order to get a double cover, we need to choose lifts $\wt e_1, \wt e_2$ and $\wt e'_1, \wt e'_2$ for the edges $e$ and $e'$. For instance, if we choose  $\wt e_1 = [u^+,v^-], \wt e_2 = [u^-, v^+], \wt e'_1 = [w^+,z^-] , \wt e'_2 = [w^-, z^+]$ we obtain the graph seen on the right side of Figure  \ref{fig:dumbbellCovers}. If, on the other hand, we choose $\wt e_1 = [u^+,v^-], \wt e_2 = [u^-, v^+], \wt e'_1 = [w^+,z^+] , \wt e'_2 = [w^-, z^-]$, we obtain the graph seen on the left  side of the same figure.

\begin{figure}
    \centering
    \includegraphics[width=0.4\textwidth]{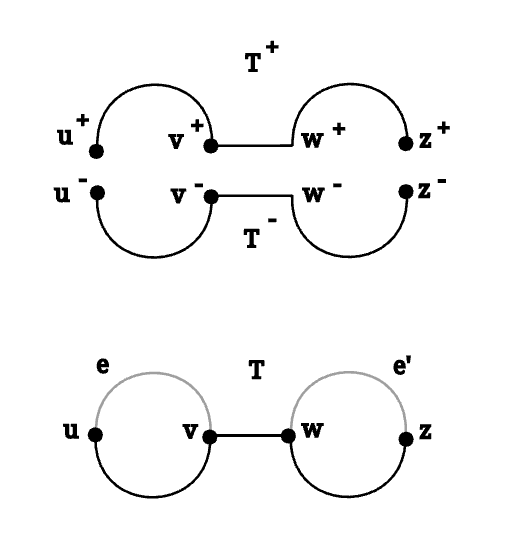}
    \caption{Spanning trees for the dumbbell graph and its double covers}
    \label{fig:spanningTrees}
\end{figure}
\end{example}

We finish this subsection with two important facts. 
Here is an important fact that you will prove in Exercise \ref{sec:PrymExercises}.\ref{exer:genusDoubleCover}. 
\begin{lemma}\label{lem:genusOfCovers}
The genus of $\widetilde\Gamma$ is $2g-1$. 
\end{lemma}

\bigskip 

\subsection{Prym varieties}\label{sec:Pryms}
Let
\[
\pi:\widetilde\Gamma\to\Gamma
\]
be a free double cover. Then we have an induced map on divisors  which sends  $\sum_i a_i y_i$ to $\sum_i a_i \pi(y_i)$. This map is well defined on divisor classes (that is, if $\wt D$ and $\wt D'$ are equivalent then so are their images) and therefore induces a surjective map between Jacobians
\[
\pi_*:\Jac(\widetilde\Gamma)\to\Jac(\Gamma).
\]
The map $\pi_*$ is known as the \emph{norm} map. 

From Theorem \ref{thm:JacobianStructure} and Lemma \ref{lem:genusOfCovers}, we know that $\Jac(\widetilde\Gamma)\simeq \RR^{2g-1}/\ZZ^{2g-1}$ and $\Jac(\Gamma)\simeq \RR^{g}/\ZZ^{g}$. The kernel of the norm map is a subgroup of $\Jac(\widetilde\Gamma)$. It should  be plausible that the kernel has dimension $g-1$. In fact, it turns out that the kernel consists of  two connected components, each of which is homeomorphic to $\RR^{g-1}/\ZZ^{g-1}$. 

\begin{definition}\label{def:Prym}
Each connected component of   $\ker(\pi_*)$ is referred to as a \emph{Prym variety}. The component that contains $0$ is a subgroup of the Jacobian called the \emph{even} Prym variety, denoted  $\Prym_0(\widetilde\Gamma/\Gamma)$. The other connected component is  called the \emph{odd} Prym variety and is denoted $\Prym_1(\widetilde\Gamma/\Gamma)$.
\end{definition}

The names even and odd will be justified in Lemma \ref{lem:evenOdd} below. 
We will occasionally use $\Prym(\widetilde\Gamma/\Gamma)$ to denote either the even or odd Prym variety when the sign is known from context. We use the term \emph{Prym divisor} to refer to divisors in the kernel of the norm map. 
In the literature, the term Prym variety often refers strictly to $\Prym_0(\widetilde\Gamma/\Gamma)$. However, for our purposes it will be convenient to be more agnostic with those terms.

\begin{example}\label{ex:somePrymdivisors}
 Consider the double cover $\pi:\wt\Gamma\to\Gamma$ as in the right hand side of Figure \ref{fig:dumbbellCovers}.
 For any point $p\in\wt\Gamma$, the divisor $p-\iota p$ is clearly a Prym divisor.  In general, $D-\iota D$ is a Prym divisor for any divisor $D$ on any  free double cover. But not all Prym divisors are of this form. For instance, if  $p$ and $q$ are any points such that $\pi(p)$ and $\pi(q)$ belong to the bridge of $\Gamma$ then $p-q, p-\iota q, \iota q - p, \iota p - \iota q$ are all Prym divisors since their norms are equivalent to 0.
\end{example}

\begin{example}
Let  $\pi:\wt\Gamma\to\Gamma$ be the double cover of the cycle by itself as in Example \ref{ex:doubleCoverCycle}. It is straightforward to check that every divisor of the form $p_1 + p_2 + \cdots p_{2k} - \iota(p_1 + p_2 + \cdots p_{2k})$ is equivalent to $0$. On the other hand, for any point $p$, the divisor $p-\iota p$ is equivalent to the divisor $\eta$ with a chip at $0$ and an anti-chip at $\frac{1}{2}$. It follows that every divisor of the form $p_1 + p_2 + \cdots p_{2k+1} - \iota(p_1 + p_2 + \cdots p_{2k+1})$ is equivalent to $\eta$. 

\end{example}

As we saw in the examples above, if we place chips and anti-chips at points that are swapped by the involution we get a divisor in the kernel of the norm map. 
In other words, every divisor of the form $E-\iota(E)$ is in the kernel.  We refer to a divisor of that form as \emph{anti-symmetric}. 
As exhibited in Example \ref{ex:somePrymdivisors}, not all Prym divisors have this form.  However, every Prym divisor is equivalent to an anti-symmetric divisor.

\begin{lemma}\label{lem:antiSymmetric}
Suppose that $\pi_*(\wt D)\simeq 0$. Then $\wt D\simeq E-\iota(E)$ where $E$ is some effective divisor. \end{lemma}

\begin{proof}
This is a sketch of the proof, you will complete the missing details in the exercise \ref{sec:PrymExercises}.\ref{exer:antiSymmetric}.
Suppose that $D:= \pi_*(\wt D)\simeq 0$. Then there is a piecewise linear function $\varphi$ such that $D + \ddiv\varphi = 0$. Define a piecewise linear function $\widetilde\varphi:\widetilde\Gamma\to\RR$  as follows.  
First, choose a model for $\Gamma$ where the slopes of $\varphi$ are  constant  along each edge. At every vertex $v$, set $\widetilde\varphi(v) = \frac{\varphi(v)}{2}$. Let $e$ be an edge of $\Gamma$ with endpoints $u$ and $w$ and let $e',e''$ be the preimages of $e$ with endpoints $u',w'$ and $u'',w''$ respectively. If the slope of $\varphi$  on $e$ in the direction from $u$ to $w$ is $d$, 
define $\wt\varphi$ on $e'$ by letting it have slope $d$ from $u'$ to the midpoint of the edge and slope $0$ from the midpoint to $w'$. Similarly, let $\wt\varphi$ have slope $0$ between $u''$ and the midpoint and slope $d$ between the midpoint and $w''$. 

Now, check that the construction gives rise to a  well-defined and continuous piecewise linear  function and that, moreover, $\wt D+\ddiv(\wt\varphi)$ is anti-symmetric.


\end{proof}

Furthermore,
\begin{lemma}\label{lem:evenOdd}
The even (resp. odd) Prym variety consists of  divisor classes of the form $[E-\iota E]$ where $E$ has even (resp. odd) degree. 
\end{lemma}

\begin{proof} (sketch)
 Let $P$ be a Prym divisor.
 By Lemma \ref{lem:antiSymmetric}, we can assume that $P = F-\iota F$ for some effective divisor $F$. We need to show that $P$ is in the even Prym variety if and only if $\deg(F)$ is even. Assume first that the degree of $F$ is even. 
 If $\deg(F) = 0$ then $P$ is just $0$ and is clearly in the even Prym variety. Otherwise, $\deg(F) \geq 2$. Let $p_1,p_2$ be points in the support of $F$. Then by continuously moving the chip at $p_1$  along with the anti-chip at $\iota p_1$, we obtain a 1-parameter family of  (not necessarily equivalent) divisors (this last bit is what needs  to be made precise in order to upgrade the argument from a `sketch' to a `proof'). We move those chips in an anti-symmetric fashion to make sure that the family remains within the even Prym variety. We continue to deform the divisor until  the chip from $p_1$ reaches $\iota p_2$  and the anti-chip from $\iota p_1$ reaches $p_2$. The resulting divisor is $E-\iota E$, where $E = F - p_1 - p_2$. 
We have thus found a path from $P$ to $e-\iota E$ where $\deg(E) = \deg F-2$. By induction, there is a path from $P$ to $0$ that lies in the even Prym variety.

A similar argument shows that all the divisors of the form $F-\iota F$, where $F$ is an effective divisor of odd degree belong to the same connected component. 
Since $\Ker\pi_*$ consists of exactly two connected components, and each element of the kernel is equivalent  to a divisor of the form $E-\iota E$, those two components must be distinct and determined by the parity of the degree of $E$. 
\end{proof}

\begin{example}
In Exercise \ref{sec:PrymExercises}\ref{exer:cycle}, you will confirm that the Prym variety corresponding to  the double cover of the cycle by itself consists of two points indexed by the parity.

Now consider the divisor $D = x+y-p-q$ seen in the double cover of the dumbbell  in Figure \ref{fig:dumbbellNonSymmetric}. Then $\pi_*(D)\simeq 0$ although $D$ is not anti-symmetric. On the other hand, by moving the chips at $x$ and $y$ in unison until they reach $p$, we obtain the anti-symmetric divisor $p-q$, since $q=\iota p$. Note that this is an odd Prym divisor. In particular, if we try to move chips around while maintaining anti-symmetry, we will never reach the 0 divisor.  

\begin{figure}
    \centering
    \includegraphics[width=0.3\textwidth]{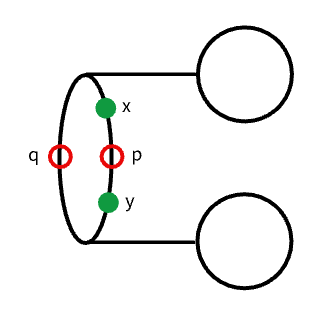}
    \caption{A Prym divisor that is not anti-symmetric}
    \label{fig:dumbbellNonSymmetric}
\end{figure}
\end{example}

Note that the even Prym variety is a group whereas the odd Prym variety is not. However,  fixing a point $p\in\widetilde\Gamma$, we have a bijection between $\Prym_0(\widetilde\Gamma/\Gamma)$ and $\Prym_1(\widetilde\Gamma/\Gamma)$ by mapping $D$ to $D+(p-\iota p)$. 
This motivates the following definition. 

\begin{definition}
Let $\pi:\widetilde\Gamma\to\Gamma$ be a double cover of metric graphs. The \emph{Abel--Prym} map of degree $d$ is the map
\[
\Psi_d:\Sym^d(\widetilde\Gamma)\to\Prym_\varepsilon(\widetilde\Gamma/\Gamma)
\]
(where $\varepsilon\equiv d\pmod 2$)
given by $\Psi_d(p_1+\cdots +p_d) = [p_1+\cdots + p_d - \iota p_1 - \ldots - \iota p_d]$. We often omit the degree when known from context and refer to $\Psi$ simply as the Abel--Prym map. 
\end{definition}

In other words, $\Psi$ is the Prym-theoretic version of the Abel--Jacobi map. 
From Lemma \ref{lem:antiSymmetric}, every element of the Prym variety is in the image of an  Abel--Prym map of some degree. Since the dimension of the Prym variety is $g-1$, it seems reasonable to believe that the Abel--Prym map of degree $g-1$ will be surjective. 
This turns out to be correct and will be discussed in the next section.

\begin{remark}\label{rem:PrymTorelli}
Prym varieties can be assigned a principal polarization making them a PPTAV (see Remark \ref{rem:Torelli}). 
We therefore obtain a map 
\[
\mathcal{R}_g^{\trop} \to \mathcal{A}^{\trop}_{g-1}
\]
from the moduli space of harmonic double covers to the moduli space of principally polarized tropical abelian varieties of dimension $g-1$. The index $g$ in $\mathcal{R}_g^{\trop}$ stands for the fact that the graphs being covered have genus $g$. This map is known as the \emph{tropical Prym--Torelli} map. 
In algebraic geometry, it is known that the closure of the image of the Prym--Torelli map  contains the image of the usual Torelli map. That is, every Jacobian is  the limit of Prym varieties. A stronger property holds in tropical geometry: every tropical Jacobian is a tropical Prym variety. A sketch of the proof is given in Exercise \ref{sec:PrymExercises}.\ref{exer:PrymSurjective}.

Unlike the standard Torelli map, the  Prym--Torelli map does not extend to the boundary of the moduli space \cite{FriedmanSmith_Prym, ABH_PrymDegenerations, CGHR_Prym}. In fact, finding reasonable compactifications of 
$\mathcal{R}_g$ and $ \mathcal{A}_{g-1}$ such that the Prym--Torelli map extends is a major open problem. On the other hand, it is worth noting  the following remarkable special case. When $g=6$, the Abel--Prym map is dominant onto its image and finite. The degree of the map is $27$ 
and there is a natural correspondence between the fibre and lines on a cubic! (see \cite{Donagi_Tetragonal} for this beautiful construction and \cite{RoehleZakharov_ngonal} for ongoing work towards a tropical version). 

\end{remark}

\begin{remark}
Prym varieties behave well under tropicalization \cite[Theorem A]{Len_Ulirsch_Skeletons}. That is, if $\pi:\wt\Gamma\to\Gamma$ is the tropicalization of an \'etale double cover $f:\wt X\to X$ of curves, then 
\[
\Trop(\Prym_0(\wt X/ X)) = \Prym_0(\wt\Gamma / \Gamma). 
\]

\end{remark}

\subsection{Prym--Brill--Noether theory}
Just as Brill--Noether theory studies special subvarieties in the Jacobian, its Prym theoretic version studies subvarieties in the Prym variety.  However, since it doesn't makes sense to talk about the rank of divisors of degree $0$, the we work in a certain natural translation of the Prym variety.  

\begin{definition}
Let $f:\wt{X}\to X$ be a double cover of metric graph or an algebraic curve and fix an integer $r$. The \emph{Prym--Brill--Noether variety} of $X$, is

\[
V^r(X,f) = \{[D]\in\Jac([\wt X]\, \big\vert\, f_*(D) = K_{X}, r(D)\geq r, r(D)\equiv r(\mod 2)\}.
\]
\end{definition}

The central question in Prym--Brill--Noether theory is for the dimension of this variety. This is known for curves that are general in moduli \cite{Bertram_existenceforPrymspecialdivisors, Welters_GiesekerPetri}. 
\begin{theorem}
Let $r$ and $g$ be integers such that $g-1-\binom{r+1}{2} \geq 0$. If the Prym--Brill--Noether variety is non-empty then 
\[
\dim V^r(X,f) \geq g-1-\binom{r+1}{2}.
\]
Conversely, if $f$ is general in the moduli of unramified double cover, then $\dim V^r(X,f)= g-1-\binom{r+1}{2}$.
\end{theorem}

Geometric techniques similar to the ones discussed in Section \ref{sec:BrillNoether} can be used to recover the converse part of the theorem. Moreover, they lead to new bounds for curves that are special in moduli  \cite[Theorem B]{Len_Ulirsch_Skeletons} and \cite[Theorem A]{CLRW_PBN}).

\section {The structure of Prym varieties}\label{sec:PrymStructure}
Recall that the tropical Jacobian has a tiling, in which  maximal cells are indexed by spanning trees. More precisely, 
every spanning tree corresponds to a cell of the Jacobian parameterizing divisors of degree $g$ supported on the complement of the tree. In this section, we will introduce an analogous result for Prym varieties. To that end, we need to understand the image and the fibres of the Abel--Prym map.
First we present a Prym-theoretic analogue of Theorem \ref{thm:AJsurjective}.

\begin{theorem}\label{APsurjective}
The Abel--Prym map of degree $g-1$ is surjective.
\end{theorem}
In other words, the representative 
$E-\iota E$ of a Prym divisor may be chosen so that $E$ is an effective divisor of degree $g-1$.  Note that, while this result is analogues to Theorem \ref{thm:AJsurjective}, their proofs are quite different since there is no known Prym-theoretic analogue of the Riemann--Roch theorem.

Now that we know the degree at which we can represent Prym divisors, the question remains as to their structure. Recall that the Abel--Jacobi map is injective (and in particular finite) on divisors supported on the complements of spanning trees. Is there an analogous statement for Prym divisors? 
In other words, is there a combinatorial condition ensuring that a divisor has a unique
 \emph{anti-symmetric} representative? 

\begin{example}\label{ex:PrymSameEdge}
We begin with a simple example. Let $\pi:\wt\Gamma\to\Gamma$ be any double cover, and let  $P = x + y - \iota x - \iota y = \Psi(x+y)$, where $x$ and $y$ are in the interior of the same edge $e$ of $\wt\Gamma$. We claim that the Abel--Prym map is not injective at $P$ and, in fact, not even finite. To that end, we will show that there is an entire family of anti-symmetric divisors $Q$ such that $\Psi(Q)=\Psi(P)$. Indeed, for any $\varepsilon>0$ small enough, consider the divisor obtained by 
 moving the chips at $x$ and $y$ and at the same time the anti-chips at $\iota x$ and $\iota y$ a  distance $\varepsilon$ towards each other. This new divisor is linearly equivalent to $P$ but is the image of a different point of $\Sym^2(\wt\Gamma)$. 
 
 In particular,
$\Psi$ is not  finite on cells of the form $e\times e$ of the symmetric product. 
A similar idea works for cells of the form $e\times \iota e$, namely where $x$ and $y$ are in edges that are swapped by the involution (see Exercise \ref{sec:PrymExercises}.\ref{ex:PrymOppositeEdges}).
\end{example}

\begin{example}\label{ex:PrymFinite}
Now, consider the Prym divisor $p-\iota p$ on the double cover of the dumbbell graph, shown in Figure \ref{fig:dumbbellTwoRepresentatives}. In this case, if we slightly perturb the chip and anti-chip in a symmetric way, we \emph{never} get an equivalent divisor. We conclude that the Abel--Prym map is locally injective at $p$. However, it does \emph{not} follow that the map is globally injective. Indeed, one can check that $p-\iota p$ is linearly equivalent to $q-\iota q$, so $\Psi(q) = \Psi(p)$.

\begin{figure}
    \centering
    \includegraphics[width=.15\linewidth]{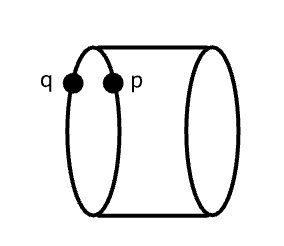}
     \caption{Two points on the double cover of the Dumbbell}
    \label{fig:dumbbellTwoRepresentatives}
\end{figure}\end{example}

In order to investigate the phenomena seen in the examples above, we make the following definition. 

\begin{definition}
Let $\pi:\widetilde\Gamma\to\Gamma$ be a free double cover and let $\Gamma'$ be a subgraph of $\Gamma$. We say that $\Gamma'$ is \emph{relatively-connected}
if the pre-image of every connected component of $\Gamma'$ is connected. If, moreover, each component of $\Gamma'$ has genus $1$ then we refer to it as a \emph{relative spanning tree}.   
\end{definition}

\begin{remark}
A relative spanning tree is also known in the literature as an \emph{odd genus 1 decomposition}. 
\end{remark}

\begin{example}
Let $x,y$ be as in Example \ref{ex:PrymSameEdge} and let $\Gamma'$ be the subgraph of $\Gamma$ obtained by removing $\pi(x)$ and $\pi(y)$. Then $\Gamma'$ consists of two connected components: the segment $J$ between $\pi(x)$ and $\pi(y)$ and its complement. The preimage of $J$ in $\wt\Gamma$ consists of the two disjoint segments between $x$ and $y$ and between $\iota x$ and $\iota y$ respectively. It follows that $\Gamma'$ is not relatively connected. 

On the other hand, if we remove the point $\pi(p)$ from the dumbbell graph in Example  \ref{ex:PrymFinite}, we obtain a single connected component. Its preimage is obtained from the graph in Figure \ref{fig:dumbbellTwoRepresentatives} by removing $p$ and $\iota p$, which is connected. It follows that the complement of $\pi(p)$ is relatively connected. 
\end{example}

As the last example and the terminology suggest, relatively connected sets and  relative spanning trees play a similar role for Prym varieties as connected sets and spanning trees play for Jacobians. This will be made precise in Theorem \ref{thm:APfinite}.
For the next lemma, recall that the genus of a non-connected graph equals 
the sum of the genera of the connected component minus the number of components plus $1$ (this follows from Equation \ref{eq:genus}).

\begin{lemma}
Let $E_1,\ldots,E_d$ be a collection of distinct open edges of $\Gamma$ and suppose that  $\Gamma' = \Gamma\setminus\{E_1,\ldots,E_d\}$  is  relatively connected. Then $d\leq g-1$. Furthermore,  $\Gamma'$ is a relative spanning tree if and only if $d=g-1$. 
\end{lemma}

\begin{proof}
Since removal of an edge reduces the genus by $1$, the genus of $\Gamma'$ is $g-d$. 
Now, suppose that $\Gamma'$  consists of $k$ connected components with genera $g_1,g_2,\ldots, g_k$. Recall that a double covering a graph of genus $h$ has genus $2h-1$. 
In particular, a graph of genus $0$ cannot have connected double covers.
Since $\Gamma'$ is relatively connected, it follows that  $g_i>0$ for all $i$. 
Therefore, the genus of $\Gamma'$ is  $\sum_{i=1}^k g_i - k + 1\geq \sum_{i=1}^k 1 - k + 1 = k-k + 1 =  1$. As the genus also equals $g-d$, we have $d\leq g-1$.

Now, if $\Gamma'$ is a relative spanning tree then, by definition, $g_i=1$ for each $i$, so $g-d = k-k+1 = 1$ and $d=g-1$. If, on the other hand, $d=g-1$ then $\sum_{i=1}^k (g_i - 1) + 1 = g-d =1$. Since each $g_i\geq 1$ for each $i$, the equality can only be satisfies if $g_i=1$ for each $i$. In particular,  $\Gamma'$ is a relative spanning tree. 
\end{proof}

We can now phrase one of the main results. 

\begin{theorem}\cite[Theorem 4.1]{LenZakharov_Kirchhoff}\label{thm:APfinite}
Let $\pi:\wt\Gamma\to\Gamma$ be a free double cover and let $\wt e_1,\ldots,\wt e_d$ be distinct edges of $\wt\Gamma$.  Denote $e_i = \pi(\wt e_i)$ for each $i$ and let   $\Gamma' = \Gamma\setminus\{e_1,\ldots,e_d\}$. Finally, denote $\sigma = \wt e_1\times\cdots\times \wt e_d$
the corresponding cell of $\Sym^{d}(\wt\Gamma)$. 
Then the Abel--Prym map of degree $d$ is finite at $\sigma$ if and only if $\Gamma'$ is relatively connected. 
\end{theorem}

Note that the theorem only states that the map is finite, but not necessarily injective. 

\begin{example}\label{ex:APfinite}
Let $\pi:\wt\Gamma\to\Gamma$ be the double cover of the dumbbell as in the right side of Figure \ref{fig:dumbbellCovers}. If $p\in\wt\Gamma$ is   such that $\pi(p)$ belongs to the bridge then the complement of $\pi(p)$ consists of two connected components. The preimage of each of them in $\wt\Gamma\setminus\{p,\iota p\}$ is connected, so $\Gamma\setminus\{\pi(p)\}$ is relatively connected. Theorem \ref{thm:APfinite} therefore implies that the Abel--Prym map is finite at $p$. 

Now let $\eta:\wt\Lambda\to\Lambda$ be the double cover seen in Figure \ref{fig:thetaCovers}. If $q$ is belongs to the edge $\alpha'$ then $\wt\Lambda\setminus\{q,\iota q\}$ is connected. It follows that $\Lambda\setminus\{\eta(q)\}$ is relatively connected, so the Abel--Prym map is finite at $q$. 
But if $r$ belongs to $\beta'$ then $\Lambda\setminus\{r,\iota r\}$ is not connected. Since $\Lambda\setminus\{\eta(r)\}$ is connected, it follows that it is not relatively connected. Theorem \ref{thm:APfinite} now implies that the Abel--Prym map is not finite at $r$. We can also verify that explicitly by perturbing $r$ and $\iota r$ simultaneously.

\end{example}

In general, determining the local degree of the Abel--Prym map is a hard problem. However, the case of degree $g-1$ is more well-behaved and can be described combinatorially. 
First, we need the following definition. 

\begin{definition}
Let $\pi:\wt\Gamma\to\Gamma$ be a free double cover and suppose that $P$ is a Prym divisor of the form $P=E-\iota E$, where $E$ is an effective divisor supported on the interior of edges. If $E$ does not have multiplicities (namely $E(x)\leq q$ at every point $x$) and $\Gamma\setminus\pi_*(E)$ is relatively connected then the \emph{weight} of $P$, denoted $wt(P)$ is the number of connected components of  $\Gamma\setminus\pi_*(E)$. Otherwise, the weight of $P$ is $0$. 
\end{definition}

\begin{theorem}\label{thm:degAP}\cite[Theorem B]{LenZakharov_Kirchhoff}
Let $P$ be a point of $\Prym_{g-1}(\wt\Gamma/\Gamma)$ whose preimage under the Abel--Prym map consists of divisors   $E_1,\ldots,E_k$ supported on the interior of edges. Denote $P_i = E_i-\iota E_i$ for each $i$.
Then
\[
\sum_{i=1}^k wt(P_i) = 2^{g-1}.
\]
\end{theorem}

Note that if a cell of $\Sym^{g-1}(\wt\Gamma)$ classifies divisors that are not fully supported on the interior of edges, then its dimension is strictly smaller than $g-1$.  Therefore, the Prym variety contains a dense open subset of divisor classes that satisfy the conditions of Theorem \ref{thm:degAP}. The theorem therefore implies that the generic Prym divisor  has $2^{g-1}$ anti-symmetric representatives, counted with their weights as multiplicities.

\begin{example}
In Example \ref{ex:APfinite}, the graph $\Gamma\setminus\{\pi(p)\}$ has two connected component, so $p-\iota p$ has weight $2$. Since the genus $\Gamma$ is $2$, it follows from Theorem \ref{thm:degAP} that $p-\iota p$ does not equal any other Prym divisor. If, on the other hand, $p$ is  as in Figure \ref{fig:dumbbellTwoRepresentatives} then $p-\iota p$ has weight $1$. So this time, it follows from Theorem \ref{thm:degAP} that $p-\iota p$ must have another representative. As we saw in Example  \ref{ex:PrymFinite}, the other representative is $q-\iota q$. 

Note that, in general, there is no guarantee that the different representatives will satisfy a nice symmetry with respect to one another or even have the same weight (see \cite[Example 5.3]{LenZakharov_Kirchhoff}).  
\end{example}

\begin{remark}
 $\Prym(\wt\Gamma/\Gamma)$ has a metric structure induced from the symmetric product. In that case, the Abel--Prym map is piecewise linear whose determinant at each cell coincides with the weight of points in its interior. As a result, the Abel--Prym map $\Psi_{g-1}$ is a harmonic map of polyhedral complexes of degree $2^{g-1}$.

Using the Abel--Prym map and the polyhedral structure of the symmetric product, we can naturally assign a volume to the cells of the Prym variety.  
Theorem \ref{thm:degAP} therefore describes the volume of the Prym variety as a weighted sum 
over all the relative spanning trees. This result is known as the \emph{tropical Prym--Kirchhoff matrix-tree theorem}. 
In the special case where all edge lengths are $1$, the formula was known prior to the discovery of tropical Prym varieties in the language of critical groups (see \cite[Proposition 9.9]{reiner2014critical}).
\end{remark}

\pagebreak
\appendix
\section{Chip-Firing Exercises}\label{sec:chipFiringExercises}


\begin{enumerate}

    \item Check whether the following divisors are equivalent to 0.
    
    \begin{figure}[h]
     \centering
     \begin{subfigure}[b]{0.2\textwidth}
         \centering
         \includegraphics[width=0.6\textwidth]{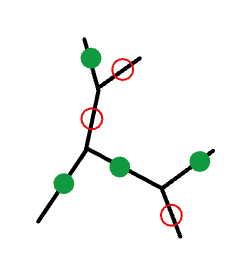}
     \end{subfigure}
     \hfill
     \begin{subfigure}[b]{0.2\textwidth}
         \centering
         \includegraphics[width=0.6\textwidth]{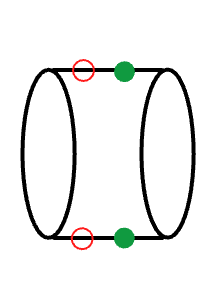}
     \end{subfigure}
     \hfill
     \begin{subfigure}[b]{0.2\textwidth}
         \centering
         \includegraphics[width=0.6\textwidth]{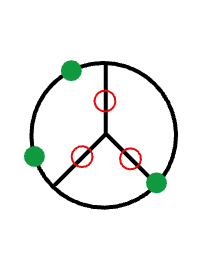}
     \end{subfigure}
     \hfill
     \begin{subfigure}[b]{0.2\textwidth}
         \centering
         \includegraphics[width=0.6\textwidth]{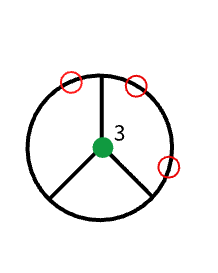}
     \end{subfigure}
     \hfill
        \caption{Four examples of divisors on graphs}
        \label{fig:Exercise1}
\end{figure}

    \item Show that when two divisors are related by Condition \ref{conditionA}, they are indeed, linearly equivalent. 
    
        \item Let $\Gamma$ be the chain of three loops and let $D=x+y+z$ be any divisor of degree $3$. Show that $D$ is equivalent to a unique divisor with a single chip on each loop. What happens if the degree of $D$ is greater than $3$? 
        
    \item Suppose that $\Gamma$ has a collection of edges $E = \{e_1,e_2,\ldots,e_k\}$ whose removal disconnects the graph. Suppose moreover that $E$ is minimal with that property (so that removing any subset of $E$ does not disconnect the graph).  
    Suppose that $D = p_1 + p_2 + \cdots + p_k$ where each $p_i$ is in the interior of $e_i$. Show that the divisor obtained by moving each $p_i$ a small distance $\varepsilon$ is equivalent to $D$.  
    
    What happens of $E$ is not minimal?

\item \label{exer:directSum} Let $\Gamma_1$ and $\Gamma_2$ be metric graphs and denote $\Gamma_1\underset{p_1,p_2}{\cup}\Gamma_2$ the graph obtained by attaching them along points $p_1\in\Gamma_1,p_2\in\Gamma_2$. Find a natural bijection between $\Jac(\Gamma_1\underset{p_1,p_2}{\cup}\Gamma_2)$ and $\Jac(\Gamma_1)\oplus\Jac(\Gamma_2)$.

    \item Find the $p$-reduced representatives of the following divisors.
    
    \begin{figure}[h]
     \centering
     \hfill
     \begin{subfigure}[b]{0.3\textwidth}
         \centering
         \includegraphics[width=0.6\textwidth]{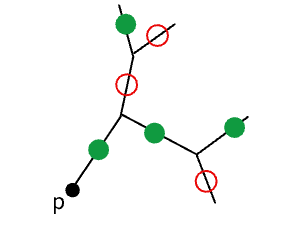}
     \end{subfigure}
     \hfill
     \begin{subfigure}[b]{0.3\textwidth}
         \centering
         \includegraphics[width=0.6\textwidth]{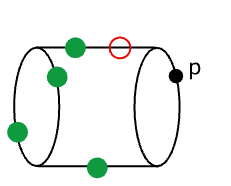}
     \end{subfigure}
     \hfill
     \begin{subfigure}[b]{0.3\textwidth}
         \centering
         \includegraphics[width=0.6\textwidth]{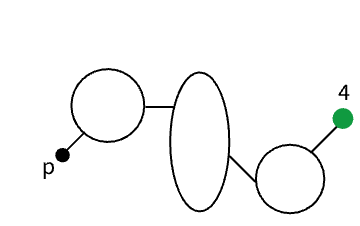}
     \end{subfigure}
     \hfill
        \caption{Divisors and a point $p$}
        \label{fig:Exercise6}
\end{figure}
 
 \item\label{exer:rigidDivisors} Let $D$ be an effective divisor  on a graph $\Gamma$, and assume that it is supported on the interior of edges and doesn't have multiplicities, namely $D(x)\leq 1$ for all $x$.
 \begin{enumerate}
     \item Suppose that $\Gamma\setminus D$ is connected. Show that  $D$ doesn't move. Namely, that there is no other effective divisor $E$ such that $D\simeq E$.
     
     \item Conversely, if $D$ doesn't move, show that   $\Gamma\setminus D$ is  connected. 
     
     \item If $\deg(D) = g$ (where $g$ is the genus of $\Gamma$), show that $D$ is rigid (namely, it is the unique effective in its class) if and only if the complement of $D$ is a spanning tree for $\Gamma$.
     
     \item If $\deg(D)>g$ show that  $D$ always moves. 
 \end{enumerate}

    \item\label{exer:reduced} Let $D$ be a $p$-reduced divisor. 
    \begin{enumerate}
        \item Suppose that $D(p)<0$. Show that $D$ is not equivalent to an effective divisor (hint: otherwise, there is a piecewise linear function $\varphi$ such that $D+\ddiv\varphi$ is effective. Now let $A$ be the set where $\varphi$ obtains its minimum).
        
        \item Suppose that $D(p)\geq 0$ and let $D'$ be an effective divisor equivalent to $D$. Show that $D'(p)\leq D(p)$.  
        
        \item Show that divisors have a unique reduced representative. That is, if $D'\simeq D$ and $D'$ is also $p$-reduced, then $D'=D$. (this part of the question does not depend on the previous parts)
        
        \item Divisors on algebraic curves.\label{exer:algebraicCurve}
\begin{enumerate}
    \item Let $p$ be a point on a curve $C$  of genus $g>0$. Show that the divisor $D=p$ has rank $0$.
Hint: otherwise, consider the map to projective space associated with D. It is a map of
degree $1$ onto $\PP^1$.

\item Let C be a curve and suppose that $p\simeq q$ for some $p\neq q$. Show that the genus of $C$ is 0.  

\end{enumerate}

    \end{enumerate}
\end{enumerate}

\newpage

\section{Rank exercises}\label{sec:rankExercises}

\begin{enumerate}

\item\label{exer:canonicalDegree} Let $\Gamma$ be a metric graph of genus $g$. Show that the canonical divisor $K_{\Gamma}$ has degree $2g-2$.

\item\label{exer:superadditive} Let $D$ and $D'$ be divisors such that $r(D)\geq 0$. Show that
$r(D+D')\geq r(D) + r(D')$. What happens when both $D$ and $D'$ have rank $-1$?  

\item\label{exer:nonTreeRank} Let $D$ be a divisor of degree $d>0$ on a graph $\Gamma$ that is not a tree. Show that $r(D) < d$.

\item  For a divisor $D$  and a point $q\in\Gamma$, denote $D_q$    the $q$-reduced divisor equivalent to $D$.
Show that a divisor $D$ has rank  1 if and only if $D_q(q)\geq 1$ for every point $q$. Is
the same true for higher ranks (where $D$ has rank $r$ if $D-p_1-p_2-\ldots - p_r$ is equivalent to effective for \emph{every} collection $p_1,\ldots,p_r$)?

\item\label{exer:hyperelliptic} 
Recall that a graph that is not a tree is called \emph{hyperelliptic} if it has a divisor of degree $2$ and rank  $1$ (see the previous question for the definition of having rank $1$). A graph is called \emph{$2$-connected} if it does not become disconnected after the removal of a single point. 

\begin{enumerate}
    \item Show that a 2-connected graph $\Gamma$ is hyperelliptic if and only if it has an involution $\tau$ (not necessarily fixed-point free) such that $\Gamma/\tau$ is a tree.  
    
    \item Explain what can go wrong when $\Gamma$ is not 2-connected. 
\end{enumerate}

\end{enumerate}

\newpage

\section{Double covers and Pryms exercises}\label{sec:PrymExercises}
\begin{enumerate}

\item\label{exer:genusDoubleCover} Suppose that $\widetilde\Gamma$ is a free double cover of a metric graph $\Gamma$ of genus $g$. Show that the genus of $\widetilde\Gamma$ is $2g-1$.

\item Let $\pi:\wt\Gamma:\to\Gamma$ be free a double cover. Show that the cover $\pi'$ obtained by contracting an edge of $\Gamma$ and the corresponding edges of $\wt\Gamma$ is a free double cover as well.

\item Sketch all the double covers of the chain of three loops and of the peace-sign graph.

\item Let $D = p-q$ be the divisor shown below (where the graph is double covering the dumbbell).    Find an anti-symmetric divisor equivalent to $D$.

\begin{figure}[h]
    \centering
    \includegraphics[width=.15\linewidth]{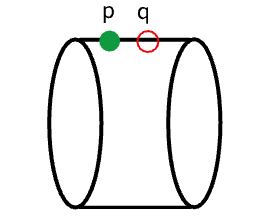}
    \caption{The divisor $D$ on the double cover of the dumbbell}
    \label{fig:dumbbellExample}
\end{figure}

\item\label{exer:antiSymmetric} \begin{enumerate}
    \item Complete the missing steps in the proof of Lemma \ref{lem:antiSymmetric}.
    \item Keeping the same notations as the in the lemma, find a bound for the degree of $E$ in terms of the genus of the graph. Theorem \ref{thm:AJsurjective} and \cite[Lemma 4.2]{Len_BNrank} could be useful here.  
\end{enumerate}

\noindent Note that a stronger version of this result follows from Theorem \ref{thm:degAP}. However, the proof given here has the advantage of being more explicit and provides more control over the resulting divisor.

\item\label{exer:cycle} Let $\pi:\Gamma\to\Gamma$ be the double cover of the cycle graph by itself as in Example \ref{ex:doubleCoverCycle}. Show that the even and odd Prym varieties each consists of a single point. 
In particular, verify Lemma \ref{lem:evenOdd}. 

\item Let $\pi:\wt\Gamma\to\Gamma$ be any of the double covers seen in Figure \ref{fig:covers}. Find an explicit bijection between the Prym variety and $\RR/\ZZ$.

\item\label{ex:PrymOppositeEdges} 
Let $P = x+y -\iota x - \iota y$, where $x\in e$ and $y\in \iota e$ for some edge $e$ of $\wt\Gamma$. 
Find an infinite family of effective divisors $z+w$ such that $\Psi(z+w) = P$. In particular, this shows that the Abel--Prym map is not finite on cells of the form $e\times \iota e$.   

\item Let $\pi:\wt\Gamma\to\Gamma$ be a free double cover where $\Gamma$ has genus $g$,  and suppose that $P = E-\iota E\in\Div\wt\Gamma$, where  the supports of $E$ and $\iota E$ do not intersect and $\deg E = d \geq g$. Show that the Abel--Prym map $\Psi_d$ is not finite at $P$. That is, that there are infinitely many effective divisors $F$ of degree $d$ such that $F-\iota F\simeq E-\iota E$. 

\item Let $\pi:\wt\Gamma\to\Gamma$ be a free double cover and suppose that $\Gamma$ is 2-connected (that is, it remains connected after the removal of a single point) and has genus greater than 1. show that the Abel--Prym map $\Psi_1$ is finite. That is, for every divisor of the form $p-\iota p$, there are only finitely many divisors $q-\iota q$ that are equivalent to it. 

\item  Give an example of a Prym divisor that has eight different Prym representatives. Find an example of a Prym divisor whose representatives have different weights.

 \item\label{exer:PrymSurjective}
 In this exercise, we  present a partial proof for the fact that every tropical Jacobian is a tropical Prym variety of some double cover 
 (cf. Remark \ref{rem:PrymTorelli}).
 As in Exercise \ref{sec:chipFiringExercises}.\ref{exer:directSum}, we use $\Sigma_1\underset{x_1,x_2}{\cup}\Sigma_2$ to denote the attachment of two graphs along a point.

\begin{enumerate}
    \item Suppose that $\widetilde{\Gamma}\to \Gamma$ is a double cover and  let $\Lambda$ be another metric graph. Fix points $p\in\Gamma$ and $q\in\Lambda$ and let $p_1,p_2$ be the preimages of $p$ in $\wt\Gamma$.
     Consider the double cover 
     \[
     \pi':\widetilde{\Gamma}\underset{p_1,q_1}{\cup}\Lambda_1\underset{p_2,q_2}{\cup}\Lambda_2\to \Gamma\underset{p,q}{\cup}\Lambda,
     \]
     where $\Lambda_1$ and $\Lambda_2$ are two isomorphic copies of $\Lambda$ and $q_1,q_2$ are their points corresponding to $q$.
     Find a natural bijection between the Prym variety of this double cover and $\Prym(\widetilde{\Gamma}/\Gamma)\oplus \Jac(\Lambda)$.
     
     \item Given a graph $\Gamma$, find a double cover whose Prym variety has a bijection with $\Jac(\Gamma)$. 
\end{enumerate}

\medskip
\noindent For the proof to be complete, we would  need to properly define principally polarized tropical abelian varieties and isomorphisms thereof.    
\medskip

\item\label{exer:PrymHyperelliptic} This question revolves around hyperelliptic graphs (see Exercise \ref{sec:chipFiringExercises}.\ref{exer:hyperelliptic} for more details). These examples were initially worked out by Zo\"e Gemmell and Dyson Yang as part of the St Andrews undergraduate summer research program.
Let $\pi:\wt\Gamma\to\Gamma$ be a free double cover. 
\begin{enumerate}
    \item Show that if $\wt\Gamma$ is hyperelliptic  then so is $\Gamma$.
Furthermore, show that the hyperelliptic involution of $\wt\Gamma$ commutes with the involution $\iota$ coming from the double cover. 
    
    \item Find an example where $\Gamma$ is hyperelliptic but $\wt\Gamma$ is not.  
    
    \item Show that, if $\wt\Gamma$ is hyperelliptic, then the Abel--Prym map $\Psi_1$ is never injective. 
    
    \item Find an example where $\Gamma$ is not hyperelliptic but the Abel--Prym map is not finite. 
    
    \item Find an example where $\wt\Gamma$ is not hyperelliptic and the Abel--Prym map is finite but not injective (this is in contrast with the analogous situation in algebraic geometry, see \cite[Corollary A.15]{LenZakharov_Kirchhoff}). 
    

\end{enumerate}


    
    
     
     
     


\end{enumerate}

\newpage

\section{Open problems}\label{sec:openProblems}
\begin{enumerate}

    

\item Let $\pi:\wt\Gamma\to\Gamma$ and $\eta:\Lambda'\to\Lambda$ be two free double cover, and let $\pi\oplus\eta$ be the double cover obtained by attaching $\Gamma$ to $\Lambda$ along a single point and attaching $\wt\Gamma$ to $\wt\Lambda$ along the two preimages of the point. Describe the Prym variety of this double cover (cf. Exercise \ref{sec:PrymExercises}.\ref{exer:PrymSurjective} for a similar but  simpler case).

\item Many of the results and definitions described in Section \ref{sec:PrymStructure} assume that divisors are supported on the interior of edges. Can they be extended to any divisors? 

\item Let $G$ be a model for a metric graph $\Gamma$ and let $P$ be a Prym divisor supported on the vertices of $G$. Is $P$ linearly equivalent to a divisor of the form $E-\iota E$ where $E$ is supported on the vertices of $G$? If so, what is the optimal upper bound for the degree of $E$?

\item\label{exer:hyperellipticOpen} In Exercise \ref{sec:PrymExercises}.\ref{exer:PrymHyperelliptic} we exhibited a tropical phenomena that is different from its  algebro-geometric analogue. In all the examples that I can come up with, the target graph $\Gamma$ is not 2-connected. What happens otherwise? That is, suppose  that $\wt\Gamma$ is \emph{not} hyperelliptic, $\Gamma$ is $2$-connected, and that the Abel--Prym map $\Psi_1$ is finite. Is the Abel--Prym map  injective?  

\item More generally, find the degree of the maps 
\[
\Sym^d(\widetilde{\Gamma})\to\Prym_{\varepsilon}(\widetilde{\Gamma}/\Gamma),
\]
where $\varepsilon$ is either 0 or 1 as appropriate. Note that for $d< g-1$ the answer may depend on the graph (one case of that already occurs in  the previous question).

\item Theorem \ref{thm:degAP} shows that every Prym divisor has $2^{g-1}$ anti-symmetric representatives, counted with appropriate multiplicities, but does not offer a way of finding them. Describe an algorithm for finding all the representatives.

\item  Is there a Prym-theoretic version of reduced divisors or Dhar's burning algorithm?

\item Work out the analogous story (to everything that we have done in these notes) for triple covers and beyond. See \cite{LenUlirschZakharov_AbelianCovers} for a classification of abelian covers of any degree and \cite{NaranjoOrtegaSpelta_cyclicPrym} for a recent treatment of Prym varieties of cyclic covers of algebraic curves.

\end{enumerate}

\newpage

\bibliographystyle{alpha}
\bibliography{PrymBib}

\end{document}